\documentclass{amsart}
\usepackage{amssymb, amsmath, amsfonts}
\usepackage{tikz-cd}
\usepackage{subcaption}
\usepackage[bookmarks=true, linktocpage=true,
bookmarksnumbered=true, breaklinks=true,
pdfstartview=FitH, hyperfigures=false,
plainpages=false, naturalnames=true,
colorlinks=true, pagebackref=true,
pdfpagelabels]{hyperref}
\hypersetup{
	linkcolor = blue,
	citecolor = blue,
	urlcolor = blue,
	colorlinks = true,
}

\calclayout

\newtheorem{theorem}{Theorem}[section]
\newtheorem{theorem*}{Theorem}[section]
\newtheorem{lemma}[theorem]{Lemma}
\newtheorem{corollary}[theorem]{Corollary}
\newtheorem{proposition}[theorem]{Proposition}
\theoremstyle{definition}
\newtheorem{definition}[theorem]{Definition}

\newcommand{\R}{\mathbb{R}}              
\newcommand{\Z}{\mathbb{Z}}              
\newcommand{\norm}[1]{\Vert{#1}\Vert}    
\newcommand{\e}{\mathbf{e}}              
\newcommand{\f}{\mathbf{f}}              
\newcommand{\interval}{\mathbb{I}}       
\newcommand{\chains}{C_{\ast}}           
\newcommand{\cochains}{C^{\ast}}         
\newcommand{\chain}[1]{C_{#1}}           
\newcommand{\sh}{\mathfrak{sh}}          
\newcommand{\Or}{{\rm Det}}              
\newcommand{\cman}{\mathrm{cMan}}        
\newcommand{\cI}{\mathcal{I}}            
\newcommand{\Hom}{\textup{Hom}}          
\newcommand{\init}{{\rm Init}}           
\newcommand{\term}{{\rm Term}}           
\newcommand{\vertices}{{\rm Vert}}       
\newcommand{\uW}{\underline{W}}          
\newcommand{\uV}{\underline{V}}          
\newcommand{\Ginit}{G^{\mathrm{init}}}   
\newcommand{\Gterm}{G^{\mathrm{term}}}   

\newcommand{\into}{\hookrightarrow}
\newcommand{\xr}{\xrightarrow}
\newcommand{\sms}{\smallsmile}
\newcommand{\pf}{\pitchfork}
\renewcommand{\th}{^{\mathrm{th}}}

\makeatletter
\@namedef{subjclassname@2020}{%
	\textup{2020} Mathematics Subject Classification}
\makeatother

\begin{document}

\title{Flowing from intersection product to cup product}

\author[G. Friedman]{Greg Friedman}
\address{Department of Mathematics, Texas Christian University}
\email{g.friedman@tcu.edu}

\author[A. Medina-Mardones]{Anibal M. Medina-Mardones}
\address{Max Planck Institute for Mathematics, Bonn, Germany}
\email{ammedmar@mpim-bonn.mpg.de}
\address{Department of Mathematics, University of Notre, Notre Dame, IN, USA}
\email{amedinam@nd.edu}
\thanks{A.M-M. acknowledges financial support from Innosuisse grant \mbox{32875.1 IP-ICT - 1} and the hospitality of the Laboratory for Topology and Neuroscience at EPFL where part of this work developed.}

\author[D. Sinha]{Dev Sinha}
\address{Mathematics Department, University of Oregon}
\email{dps@uoregon.edu}
\thanks{D.S. acknowledges financial support from the Simons Foundation.}

\subjclass[2020]{55N45, 57R19, 57R25}
\keywords{Geometric cohomology, intersection product, cup product, vector field flow, manifolds with corners}

\begin{abstract}
	We use a vector field flow defined through a cubulation of a closed manifold to reconcile the partially defined commutative product on geometric cochains with the standard cup product on cubical cochains, which is fully defined and commutative only up to coherent homotopies. The interplay between intersection and cup product dates back to the beginnings of homology theory, but, to our knowledge, this result is the first to give an explicit cochain level comparison between these approaches.
\end{abstract}

\maketitle
\tableofcontents


\section{Introduction} \label{S: intro}

de Rham cohomology has long been lauded as a perfect cohomology theory by champions such as Sullivan \cite{sullivan1977infinitesimal} and Bott \cite{BoTu82}.
A combination of geometric underpinning and commutativity at the cochain level make it a remarkably effective tool for many applications of rational homotopy theory.
Over the integers, submanifolds and intersection in various settings provide geometrically meaningful cochains \cite{Lipy14} with a partially defined commutative 
product \cite{Joyc15, FMS-foundations}.
But the obstructions to commutativity witnessed by Steenrod operations show that intersection alone cannot capture the cochain quasi-isomorphism type multiplicatively.

In this paper we start to marry two imperfect theories, relating multiplicative structures of, on one hand, geometric cochains defined using manifolds with corners, and, on the other, standard cubical cochains.
The comparison chain map $\cI$ between these ``analog" and ``digital" presentations of ordinary cohomology of a closed manifold is defined through counting intersections of geometric cochains with a given cubulation. With the proper definitions, which we set up in detail Section \ref{S: geometric cochains}, the map $\cI$ is a quasi-isomorphism.
In the domain of this comparison map there is a natural partially defined product given by transverse intersection, whereas in the target the product structure is induced from the Serre diagonal, a cubical analogue of the Alexander-Whitney diagonal in the simplicial setting.
Both of these multiplicative structures induce the standard cup product in cohomology, but at the cochain level they are not immediately compatible.
We bind them using the flow of a vector field canonically defined using the cubulation.

\begin{figure}[h] 
	\centering
	\hfill
	\begin{subfigure}[b]{0.45\textwidth}
		\includegraphics[scale=.6]{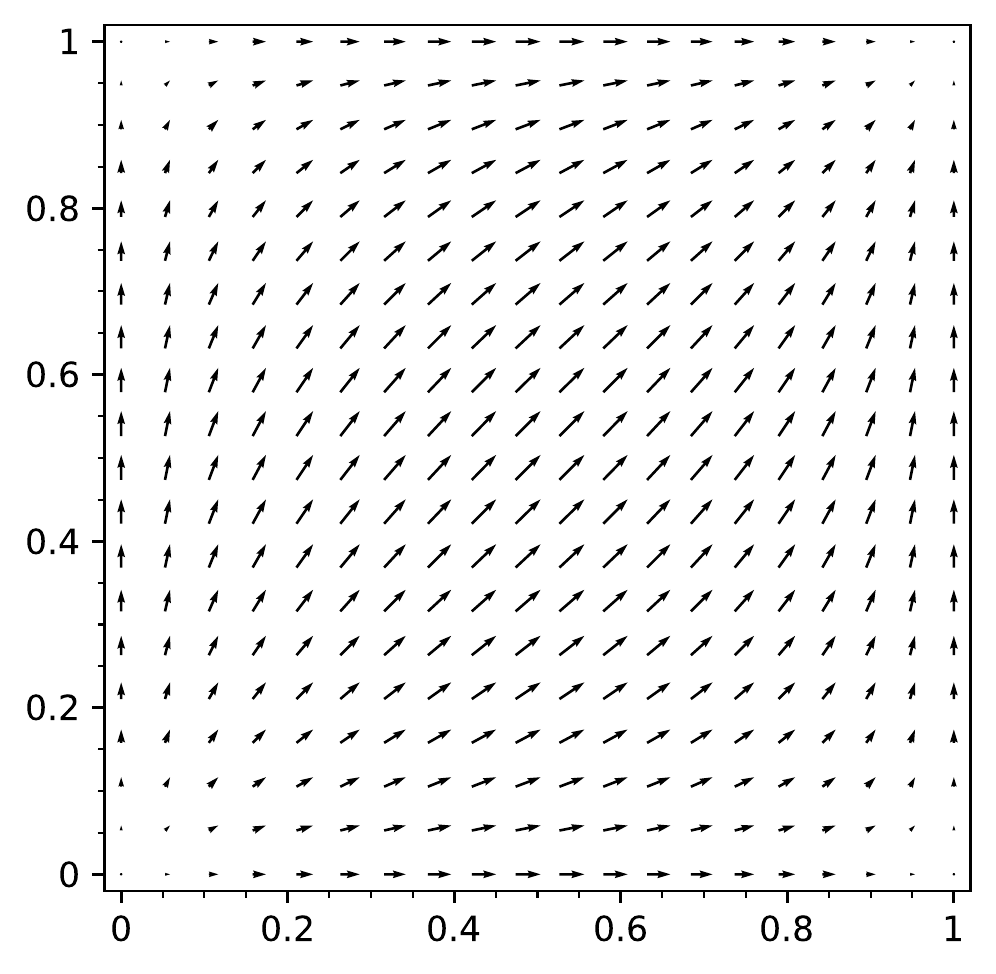}
		\caption{The logistic vector field on $\interval^2$. 
			Because logistic vector fields are consistent across cubical faces, this vector field extends to any cubulated surface.}
	\end{subfigure}
	\hspace{0.2 in}
	\begin{subfigure}[b]{0.45\textwidth}
		\includegraphics[scale=.66]{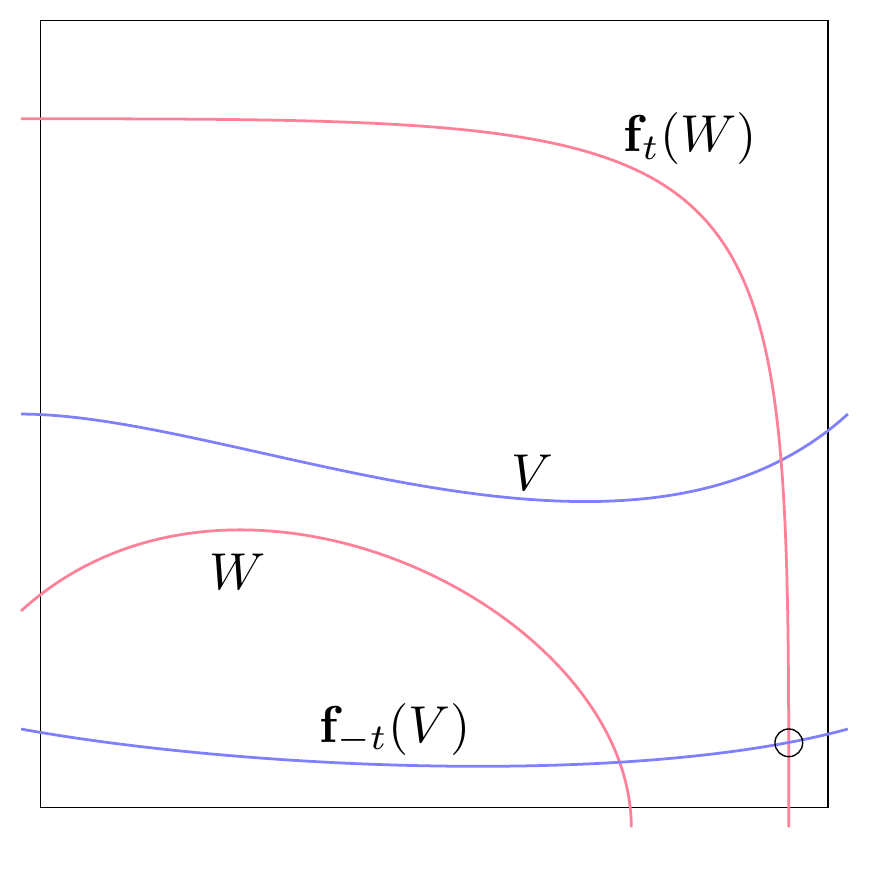}
		\caption{The intersection of $W$ and $V$ is not compatible with the corresponding cubical cup product, but that of $\f_t(W)$ and $\f_{-t}(V)$ is.}
	\end{subfigure}
	\caption{The logistic vector field and the impact of its flow on intersections.}
	\label{F: logistic}
\end{figure}

The basic idea of the construction is given in Figure~\ref{F: logistic}. The logistic vector field is pictured in part (A) of the figure with its time $t$ flow denoted by 
$\f_t$.  Part (B) of the figure illustrates the main idea of how the flow reconciles multiplications.
Here $W$ and $V$ are manifolds with corners mapping to a closed manifold $M$ which we assume cubulated, focusing the picture on a single square.
As explained in Section~\ref{S: geometric cochains}, such maps represent geometric cochains of $M$, and integer coefficients can be considered if additional (co)orientation data is included.
Geometric cochains that are transverse to the cubulation, as we are assuming $W$ and $V$ are, define cubical cochains $\cI(W)$ and $\cI(V)$ by a count of signed intersection numbers with the cubical faces.
The picture shows that with mod-two coefficients $\cI(W)$ evaluates to 1 on the bottom and left edges, while $\cI(V)$ evaluates to 1 on the left and right edges.
As explained in Section~\ref{S: cubical topology}, because the bottom edge and right edge form an ``initial-terminal'' pair of faces of the square, the Serre diagonal construction gives that $\cI(W) \sms \cI(V)$ evaluates to 1 on the square.
As $W$ and $V$ do not intersect each other, their intersection product evaluates to 0 on the square, and thus disagrees with the cup product at the cochain level.
Yet, for $t$ sufficiently large, $\f_t(W)$ and $\f_{-t}(V)$ intersect while maintaining $\cI(W) = \cI(\f_t(W))$ and $\cI(V) = \cI(\f_{-t}(V))$, now yielding agreement between the intersection and cup products at the cochain level.

Our main result is that logistic flow performs such reconciliation in general.
We write $W \times_MV$ for the fiber product of $W$ and $V$ over $M$, which gives rise to the partially defined product on geometric cochains,
defined when $W$ and $V$ are transverse.
With this notation we now present the main result of this work.
 
\begin{theorem} \label{T:main1}
	Let $M$ be a cubulated closed manifold and $W$ and $V$ two compact co-oriented manifolds with corners over $M$ which are transverse to the cubulation. Then, for $t$ sufficiently large:
	\begin{enumerate}
		\item $\f_t(W)$ and $\f_{-t}(V)$ are transverse and
		\begin{equation*}
		\cI\left(\f_t(W) \times_M \f_{-t}(V)\right) = \cI\left(\f_t(W)\right) \sms \cI\left(\f_{-t}(V)\right).
		\end{equation*}
		\item $\f_{-t}(W)$ and $\f_t(V)$ are transverse and
		\begin{equation*}
		\cI\left(\f_{-t}(W) \times_M \f_t(V)\right) = (-1)^{|W||V|} \, \cI\left(\f_t(V)\right) \sms \cI\left(\f_{-t}(W)\right),
		\end{equation*}
		where $|W||V|$ is the product of the codimensions of $W$ and $V$ over $M$.
	\end{enumerate}
\end{theorem}

It is classically known that the intersection and cup products are Poincar\'e dual at the level of homology and cohomology, but, to our knowledge, this result is the first to give an explicitly connection between these products at the cochain level. 

Turning to applications, manifold cochains have primarily been developed as a parallel to, or for application in, string topology \cite{chas1999string}, Floer theory \cite{Lipy08}, and other types of moduli questions \cite{BoJo17}.
More work needs to be done for our viewpoint to connect with these  fields, but as they stand, the results of this paper are applicable, for example, in using the bar construction on cochains to define knot invariants through induced maps on configuration spaces \cite{BCSS05, SiWa13, BCKS17}. 

Since the logistic flow interpolates between commutative and noncommutative worlds, in future work we plan to connect it to cup-$i$ products \cite{steenrod1947cocycles, medina2018axiomatic} and higher derived structures \cite{medina2020cartan, medina2020adem}.
More generally, as has been done for combinatorial cochains \cite{mcclure2003multivariate, berger2004combinatorial, medina2020prop1, medina2021cubical}, our work invites the possibility of defining \mbox{$E_\infty$ structures} on geometric cochains and the description of cohomology operations at the cochain level \cite{medina2020odd} using geometric language.
We are particularly interested in building on the work of Mandell \cite{Mand01} and others to model homotopy types of manifolds via geometric cochains.

The question of relating vector field flows to finer cochain structures has also recently arisen in mathematical physics \cite{Thor18, Tata20}, but the vector fields in \cite{Tata20} are non-continuous.
Our flow is globally smooth and thus should serve as a strong bridge between physical models, geometry, and topology.

There are two variants of Theorem~\ref{T:main1} which are likely of interest but which will not be addressed in this paper.
First, one can use simplices instead of cubes.
Working with cubulations simplifies our treatment since the logistic flow on standard cubes is given coordinate-wise by the logistic flow on the interval.
But the simplicial version of Theorem~\ref{T:main1} can be proven for simplicial cochains with the Alexander-Whitney product, using the results of this paper and the model of standard simplices as subsets of cubes with non-increasing coordinates.  
We leave the details to the interested reader.
Secondly, we conjecture that there is a version of Theorem~\ref{T:main1} in which some finite subcomplex of geometric cochains maps to a version of transverse smooth singular cochains.
Precise formulation of such a comparison map is one of the topics we plan to address in \cite{FMS-foundations}, so for now we leave this idea undeveloped.

We begin the paper by reviewing in Section~\ref{S: cubical topology} basic material on cubical structures.
We then describe geometric cochains defined using manifolds with corners, a notion that arises naturally when considering fiber products of manifolds with boundary. 
But for manifolds with corners, the boundary of a boundary is not empty, so one must impose a quotient at the cochain level to obtain a cochain complex.
In Section~\ref{S: geometric cochains}, we review the needed parts of the this theory as given in \cite{FMS-foundations} and based on the original definition of Lipyanskiy \cite{Lipy14}.
In Section~\ref{S: logistic}, we then develop logistic vector fields, for which the analysis is thankfully simple to manage.
These vector fields in a sense give a smooth extension of the cubical poset structure, the key combinatorial structure used in defining the cubical cup product.
We put everything together in Section~\ref{S: flow comparison theorem} to prove our main comparison theorem, stated above, which intuitively says that after sufficient time flow intersection yields a ring homomorphism from geometric cochains to cubical cochains.

\section*{Acknowledgments}

The authors thank Mike Miller, for pointing us to \cite{Lipy14}, and Dominic Joyce, for answering questions about his work. 


\section{Cubical topology} \label{S: cubical topology}

\subsection{Cubical complexes}

The interpolation we develop between combinatorial and smooth topology proceeds through a cubulation of a manifold -- that is, a cubical complex homeomorphic to the manifold.
Such a structure is less common than that of a triangulation, so we present basic definitions,  in a form best suited to our applications. 

For simplicial complexes, vertices can always be given a partial order that restricts to a total order on each simplex, providing a way to identify each simplex with the standard simplex.
Furthermore, when two simplices meet along a common face, the induced ordering data for that face is consistent.
Categorically, such data is reflected in the fact that every simplicial complex is the realization of some simplicial set.
There is a parallel to this in the cubical setting, namely data required to compatibly identify each $n$-cube of a cubulation with the standard $n$-cube.

We thus begin with a formulation of cubical complexes containing such extra ordering data, as well as a description of the key features of cubical structures that will be needed for the analysis of our vector field flows in Section~\ref{S: logistic}.

The \textbf{standard $n$-cube} is the subset of $\R^n$ defined by
\begin{equation*}
\interval^n = \big\{ (x_1, \dots, x_n) \in \R^n\ |\ 0 \leq x_i \leq 1 \big\},
\end{equation*}
with the standard topology but with our preferred metric being the $L^\infty$ metric.
Denote $\{1, \dots, n\}$ by $\overline{n}$.
A partition $F = (F_0, F_{01}, F_1)$ of $\overline n$ defines a \textbf{face} of $\interval^n$ given by
\begin{equation*}
\{(x_1, \dots, x_n) \in \interval^n\ |\ \forall \varepsilon \in \{0, 1\},\ i \in F_\varepsilon \Rightarrow x_i = \varepsilon\}.
\end{equation*}

We abuse notation and use the same notation for the partition and its associated face, referring to coordinates $x_i$ with $i \in F_{01}$ as \textbf{free} 
and to the others as \textbf{bound}.
The \textbf{dimension} of $F$ is its number of free coordinates, and as usual the faces of dimension $0$ and $1$ are called vertices and edges, respectively.
The set of vertices of $\interval^n$ is denoted by $\vertices(\interval^n)$.

If $F_1 = \emptyset$, then we say that $F$ is an \textbf{initial} face; if $F_0 = \emptyset$, then we say that $F$ is a \textbf{terminal} face.
Let $\init_k(\interval^n)$ be the union of initial faces of dimension $k$ and $\term_k(\interval^n)$ the union of terminal faces of dimension $k$.

\begin{figure}[!h]
	\begin{subfigure}[b]{0.3\textwidth}
		\centering
		\includegraphics[scale=.7]{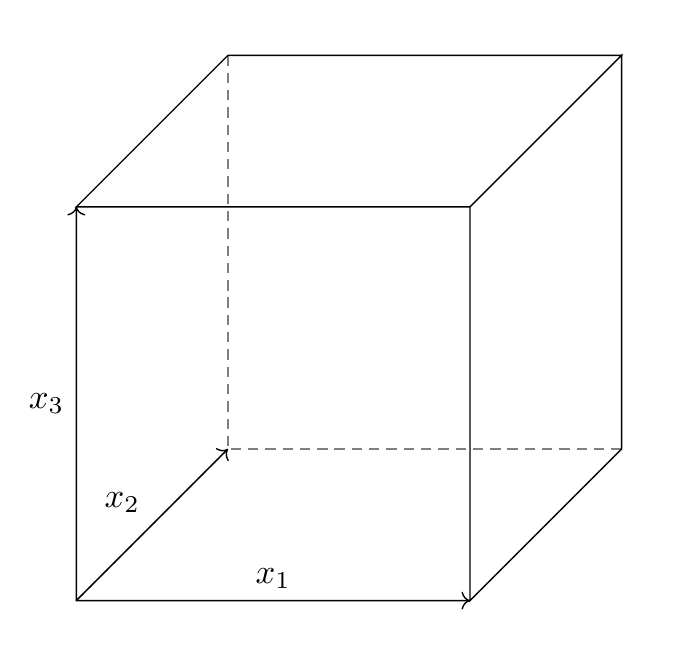}
	\end{subfigure}
	\hfill
	\begin{subfigure}[b]{0.3\textwidth}
		\centering
		\includegraphics[scale=.7]{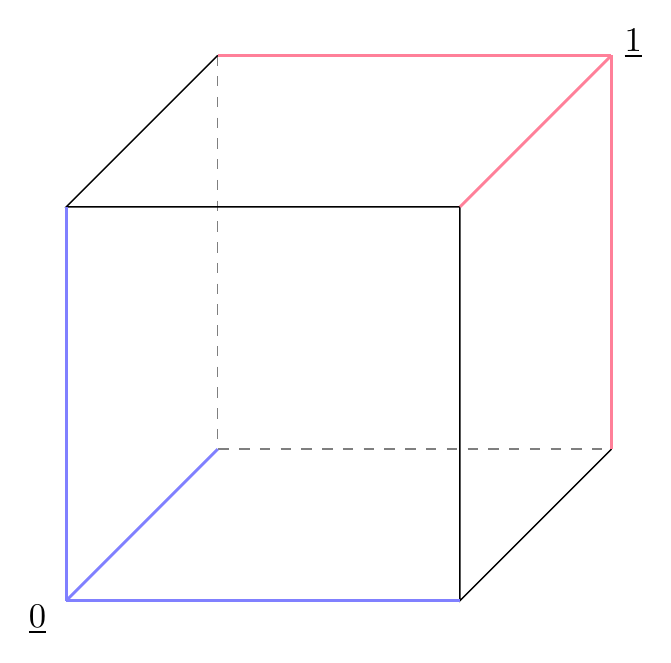}
	\end{subfigure}
	\hfill
	\begin{subfigure}[b]{0.3\textwidth}
		\centering
		\includegraphics[scale=.7]{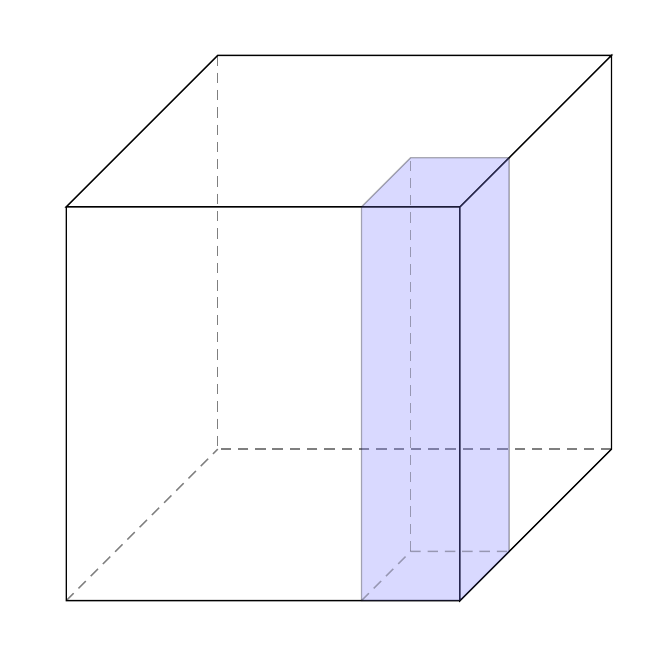}
	\end{subfigure}
	\caption{On the left we have a representation of the standard $3$-cube $\interval^3$. 
	In the center we depict $\init_1(\interval^3)$, the initial 1-dimensional faces, in blue and $\term_1(\interval^3)$, the terminal 1-dimensional faces, in red. On the right we depict an  $\epsilon$-neighborhood of $F = (\{2\}, \{3\}, \{1\})$ in the $L^\infty$ metric.}
\end{figure}

The maps $\delta_i^\varepsilon \colon \interval^{n-1} \to \interval^{n}$ are defined for $\varepsilon \in \{0, 1\}$ and $i \in \overline{n}$ by
\begin{align*}
\delta_i^\varepsilon(x_1, \dots, x_{n-1}) & = (x_1, \dots, x_{i-1}, \varepsilon, x_i, \dots, x_{n-1}),
\end{align*}
and any composition of these is referred to as a \textbf{face inclusion map}.

For $v \in \vertices(\interval^n)$ all coordinates are bound -- that is, $v_{01} = \emptyset$. Thus
 $v$ is determined by the partition of $\overline n$ into $v_0$ and $v_1$, so
we have a bijection from the set of vertices of $\interval^n$ to the power set $\mathcal P(\overline n)$ of $\overline n$, sending $v$ to $v_1$.
The inclusion relation in the power set induces a poset structure on $\vertices(\interval^n)$ given explicitly by
\begin{equation*}
v = (\epsilon_1, \dots, \epsilon_n) \leq w = (\eta_1, \dots, \eta_n) \iff \forall i,\ \epsilon_i \leq \eta_i.
\end{equation*}
We will freely use the identification of these posets.
The smallest and largest elements in $\mathcal P(\overline{n})$, which we denote $\underline{0}$ and $\underline{1}$, are the 
initial and terminal vertices. Face embedding maps induce order-preserving maps at the level of vertices.

An \textbf{interval subposet} of $\mathcal P(\overline n)$ is one of the form $[v, w] = \{u \in \mathcal P(\overline n)\ |\ v \leq u \leq w\}$ for a pair of vertices $v \leq w$. There is a canonical bijection between faces of $\interval^n$ and such subposets, associating to $[v, w]$ the face $F$ defined by $F_\varepsilon = \{i \in \overline{n}\ |\ v_i = w_i = \varepsilon\}$ for $\varepsilon \in \{0, 1\}$.

The posets $\{\mathcal P(\overline n)\}_{n \geq 1}$ play the role for cubical complexes that finite totally ordered sets play for simplicial complexes.
Recall for comparison that one definition of an abstract ordered simplicial complex is as a pair $(V, X)$, where $V$ is a poset and $X$ is a collection of subsets of $V$, each with an induced total order, such that all singletons are in $X$ and subsets of sets in $X$ are also in $X$.
We have the following cubical analogue.	

\begin{definition}\label{D:cubical}
	A \textbf{cubical complex} $X$ is a collection $\{ \sigma \}$ of finite non-empty subsets of a poset 
	$\vertices(X)$, together with, for each $\sigma \in X$, an order-preserving bijection $\iota_\sigma \colon \sigma \to \mathcal P(\overline n)$ for some $n$, such that:
	\begin{enumerate}
		\item For all $v \in \vertices(X)$, $\{v\} \in X$,
		\item For all $\sigma \in X$ and all $[u,w] \subset \mathcal P(\overline n)$ the set $\rho = \iota_\sigma^{-1}([u,w]) \in X$ and the following commutes
		\begin{equation*}
		\begin{tikzcd} [row sep = tiny, column sep= small]
		\sigma \arrow[rr, "\iota_\sigma"] && \mathcal P(\overline n) \\
		& [-5pt] {[}u,w{]} \arrow[ur, hook] & \\
		\rho \arrow[uu, hook] \arrow[rr, "\iota_\rho"'] && \mathcal P(\overline m). \arrow[ul, "\cong"'] \arrow[uu, dashed] 
		\end{tikzcd}
		\end{equation*}
	\end{enumerate}
	We refer to an element $\sigma \in X$ as a \textbf{cube} of $X$, refer to $\iota_\sigma \colon \sigma \to \mathcal P(\overline{n})$ as its \textbf{characteristic map}, 
	and refer to $n$ as its \textbf{dimension}. If $\rho \subseteq \sigma \in X$, we say that $\rho$ is a \textbf{face} of $\sigma$ in $X$. 
	We identify elements in $\vertices(X)$ with the singleton subsets in $X$, referring to them as vertices.
\end{definition}

In analogy with the usual terminology in the simplicial setting, one could call these ``ordered cubical complexes," but we only work with these and have seen little use elsewhere for the unordered version.

Consider the category defined by the inclusion poset of a cubical complex $X$ and the subcategory ${\tt Cube}$ of the category ${\tt Top}$ of topological spaces whose objects are the $n$-cubes, identified with $\interval^n$, and whose morphisms are face inclusions.
The characteristic maps of $X$ define a functor from its poset category to $\mathtt{Cube}$, and we define its \textbf{geometric realization} as the colimit of this functor.
A \textbf{cubical structure} or \textbf{cubulation} on a space $S$ is a homeomorphism $h \colon |X| \to S$ from the geometric realization of a cubical complex.
We abuse notation and write $h \circ \iota_{|\sigma|}$ simply as $\iota_\sigma$ for any $\sigma \in X$ when a cubical structure $h \colon |X| \to S$ is understood.

Our definition sits between cubical sets \cite{jardine2002cubical} and cellular subsets of the cubical lattice of $\R^\infty$ \cite{kaczynski2006computational},
analogously to the way that abstract ordered simplicial complexes sit between simplicial sets and simplicial complexes.
The geometric realization construction makes our definition and the cubical lattice definition essentially equivalent.
Just as is the case for simplicial complexes, faces in cubical complexes are completely determined by their vertices.

\begin{figure}[h]
	\begin{subfigure}{.4\textwidth}
		\centering
		\includegraphics{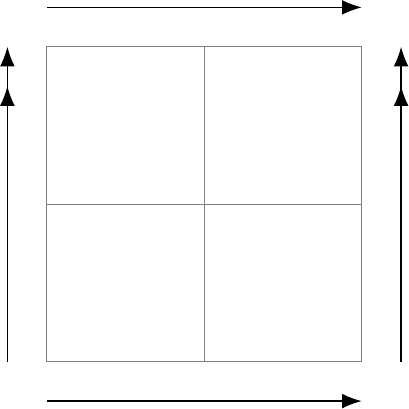}
		\caption{\textbf{Not} a cubulation of the torus}
	\end{subfigure}\qquad 
	\begin{subfigure}{.4\textwidth}
		\centering
		\includegraphics{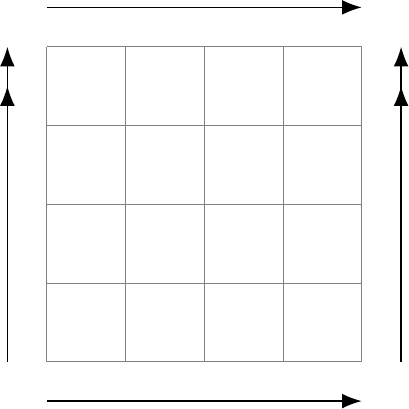}
		\caption{A cubulation of the torus}
	\end{subfigure}
	\caption{The first cellular decomposition of a torus pictured above does not represent the geometric realization of a cubical complex, as each square has the same set of vertices. On the right, each square can be coherently identified with the standard square with initial vertex in the lower left corner and final vertex in the upper right corner. Therefore, (B) depicts a cubical structure on the torus.}
	\label{F: cubical structure}
\end{figure}

The vector field flow we define on cubulated manifolds in Section~\ref{S: logistic} can be viewed as a smooth extension of the cubical poset structure. 
We refine our description of this poset structure through identifying ``previous'' and ``next'' faces in a cube.

\begin{definition} \label{D: F decomposition}
	Let $F = (F_0, F_{01}, F_1)$ be a face of $\interval^n$. The $F$-\textbf{decomposition} of $\interval^n$ is the isomorphism $\interval^n \cong F^- \times F \times F^+$ where $F^- = (F_0 \cup F_{01}, F_1, \emptyset)$ and $F^+ = (\emptyset, F_0, F_1 \cup F_{01})$.
\end{definition}

An alternate definition of $F^+$ is as the face whose initial vertex is the terminal vertex of $F$ and whose terminal vertex is $\underline{1}$, 
the terminal vertex of $\interval^n$.Similarly, $F^-$ is the face whose terminal vertex is the initial vertex of $F$ and whose initial vertex is $\underline{0}$, the initial vertex of $\interval^n$. See Figure~\ref{F: decomposition}.

\begin{figure}[h!]
	\begin{subfigure}[b]{0.35\textwidth}
		\centering
		\includegraphics[scale=.7]{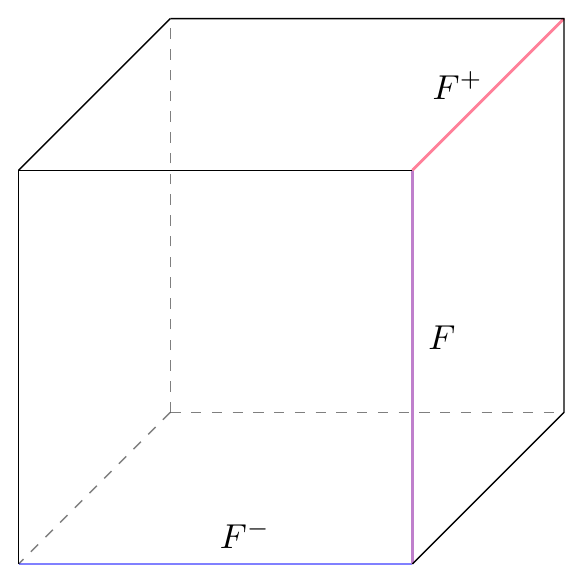}
	\end{subfigure}
	\begin{subfigure}[b]{0.35\textwidth}
		\centering
		\includegraphics[scale=.7]{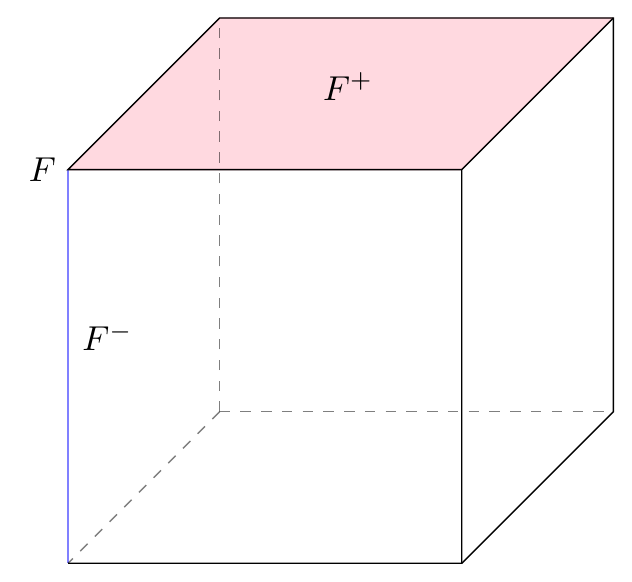}
	\end{subfigure}
	\caption{Examples of $F$-decompositions with $F = (\{2\}, \{3\}, \{1\})$ on the left and $F = (\{1,2\}, \emptyset, \{3\})$ on the right.}
	\label{F: decomposition}
\end{figure}

The special case of $F$-decompositions in which $F = v$, a vertex, merits its own consideration.

\begin{definition}\label{D: reciprocal}
	An ordered pair of faces $(F,F^\prime)$ of $\interval^n$ is said to be \textbf{reciprocal} if there exists a vertex $v$ such that $F = v^-$ and $F^\prime = v^+$. Equivalently, $(F, F')$ is reciprocal if and only if $F$ is initial and $F^\prime = F^+$, or if and only if $F'$ is terminal and $F = (F^{\prime})^-$.
\end{definition}

Consider the ordered set $\{\e_1, \dots, \e_n\}$ where $\e_i = \frac{\partial\ }{\partial x_i}$.
For any face $F$ of $\interval^n$, the ordered subset $\beta_F = \{\e_i\ |\ i \in F_{01}\}$ defines the \textbf{canonical orientation} of $F$.
We define the \textbf{shuffle sign} of $F$, denoted by $\sh(F) \in \{\pm 1\}$, to be $+1$ if the $F$-decomposition isomorphism is orientation preserving and $-1$ if not.
More explicitly, $\sh(F) = +1$ if the concatenation of the ordered sets $\beta_{F^-}$, $\beta_{F}$, and $\beta_{F^+}$ represents the same orientation as $\{\e_1, \dots, \e_n\}$, and $\sh(F) = -1$ otherwise.
This sign plays a key role in our applications, since we work over the ring of integers and this sign occurs in comparing products.

\subsection{Cubical cochains} \label{SS: cubical cochains}

We can also define an ``algebraic realization" for a cubical complex in analogy to its geometric realization. 
Let $\chains(\interval^1)$ be the usual cellular chain complex of the interval with integral coefficients.
Explicitly, $\chain0(\interval^1)$ is generated by the vertices $[\underline{0}]$ and $[\underline{1}]$, and $\chain1(\interval^1)$ is generated by the unique 1-dimensional face, denoted $[\underline{0},\underline{1}]$ in the interval subposet notation. The boundary map is $\partial [\underline{0},\underline{1}]=[\underline{1}]-[\underline{0}]$.

Let $\chains(\interval^n) = \chains(\interval^1)^{\otimes n}$, with differential defined by the graded Leibniz rule.
Given a face inclusion $\delta_i^{\varepsilon} \colon \interval^n \to \interval^{n+1}$ the natural chain map $\chains(\delta_i^{\varepsilon}) \colon \chains(\interval^1)^{\otimes n} \to \chains(\interval^1)^{\otimes n+1}$ is defined on basis elements by
\begin{equation*}
x_1 \otimes \cdots \otimes x_n \mapsto
x_1 \otimes \cdots \otimes [\underline{\varepsilon}] \otimes \cdots \otimes x_n.
\end{equation*}
Regarding a cubical complex $X$ as a functor to $\mathtt{Cube}$, we can compose it with the chain functor above to obtain a functor to chain complexes. 
The complex of \textbf{cubical chains} of $X$, denoted $C_*(X)$, is defined to be the colimit of this composition.
As one would expect, in each degree it is a free abelian group generated by the cubes of that dimension, and its boundary homomorphism sends the
generator associated to a cube to a sum of generators associated to its codimension-one faces with appropriate signs.

The \textbf{cubical cochains} of $X$ (with $\Z$ coefficients) is the chain complex $\cochains(X) = \Hom_\Z(\chains(X), \Z)$.
By abuse, we use the same notation and terminology for an element in $X$, its geometric realization in $|X|$, 
and the corresponding basis elements in $\chains(X)$ and $\cochains(X)$.

We next recall the \textbf{Serre diagonal}. Let $\Delta \colon \chains(\interval^1) \to \chains(\interval^1)^{\otimes 2}$ be defined on basis elements by	
\begin{gather*}	
\Delta([\underline{0}]) = [\underline{0}] \otimes [\underline{0}], \qquad 
\Delta([\underline{1}]) = [\underline{1}] \otimes [\underline{1}], \qquad
\Delta([\underline{0}, \underline{1}]) = [\underline{0}] \otimes [\underline{0}, \underline{1}] + [\underline{0}, \underline{1}] \otimes [\underline{1}].
\end{gather*}	
Then, let
$\Delta \colon \chains(\interval^n) \to \chains(\interval^n)^{\otimes 2}$
be the composite
\begin{equation*}
\begin{tikzcd}
\chains(\interval)^{\otimes n} \arrow[r, "\Delta^{\otimes n}"] & \left( \chains(\interval)^{\otimes 2} \right)^{\otimes n} \arrow[r, "sh"] & \left( \chains(\interval)^{\otimes n} \right)^{\otimes 2},
\end{tikzcd}
\end{equation*}
where $sh$ is the shuffle map that reorders tensor factors so that those in odd positions occur first. More explicitly, using Sweedler's notation, 
if $x_i^{(1)}$ and $x_i^{(2)}$ are defined through the identity
\begin{equation*}	
\Delta(x_i) = \sum x_i^{(1)} \otimes x_i^{(2)},
\end{equation*}	
then
\begin{equation} \label{E: Delta}	
\Delta (x_1 \otimes \cdots \otimes x_n) = 	
\sum \pm \left( x_1^{(1)} \otimes \cdots \otimes x_n^{(1)} \right) \otimes 	
\left( x_1^{(2)} \otimes \cdots \otimes x_n^{(2)} \right),
\end{equation}	
where the sign is determined by the Koszul convention.

The \textbf{cup product} of cochains $\alpha, \beta \in \cochains(X)$ is defined using the Serre diagonal as follows\footnote{
	We follow the convention for evaluation of tensor products of cochains on tensor products of chains given by $(\alpha\otimes \beta)(x\otimes y)=\alpha(x)\beta(y)$. This convention is used for defining the cup product, for example, by Munkres \cite[Section 60]{Mun84}, Hatcher \cite[Section 3.2]{Hat02}, and Spanier \cite[Section 5.6]{Spa81}. But it disagrees with the conventions in Dold \cite[Section VII.7]{Dol72}, where there is a sign coming from the Koszul convention.}:
\begin{equation*}
(\alpha \sms \beta)(c) = (\alpha \otimes \beta)\Delta(c).
\end{equation*}

We will use the following more explicit description of Serre's diagonal.

\begin{proposition} \label{P: diagonal in terms of vertices}
	The map $\Delta \colon \chains(\interval^n) \to \chains(\interval^n)^{\otimes 2}$ satisfies
	\begin{equation*}
	\Delta \big([\underline 0, \underline 1]^{\otimes n}\big) \ =
	\sum_{v \in \vertices(\interval^n)} \sh(v) \cdot v^- \otimes v^+.
	\end{equation*}
\end{proposition}

\begin{proof}
	In expression \eqref{E: Delta} each $x_i^{(1)}$ must be $[\underline 0]$ or $[\underline 0, \underline 1]$ and each $x_i^{(2)}$ must be $[\underline 0, \underline 1]$ or $[\underline 1]$. 
	Moreover, if $x_i^{(1)} = [\underline 0]$ then $x_i^{(2)} = [\underline 0, \underline 1]$, and if $x_i^{(1)} = [\underline 0, \underline 1]$ then $x_i^{(2)} = [\underline 1]$.
	Hence, in each summand of \eqref{E: Delta}, the first and second tensor factors are reciprocal faces of $\interval^n$. Conversely, each vertex of $\interval^n$ determines such a summand.
	The proposition now follows from the identification of the shuffle sign with the sign arising from applying the Leibniz rule.
\end{proof}



\section{Geometric cochains} \label{S: geometric cochains}

To specify a cubical cochain in a fixed degree is to give an integer for each and every cube in that dimension, which in practice can be an unwieldy amount of data. 
Submanifolds, which can be simple to describe in cases of interest, can encode such data through intersection.

The basic idea is classical, essentially an implementation of Poincar\'e duality at the chain and cochain level by using intersection with a submanifold in order to define a function on chains.
We implement this idea in Definition~\ref{D: intersection homomorphism}.
But there are technical obstacles to overcome in order to obtain cochain models. 
First, submanifolds alone do not capture homology and cohomology, as Thom famously realized and as can be seen in applications such as using Schubert varieties to represent cohomology of Grassmannians.
So we generalize from submanifolds to any manifold equipped with a map to our manifold in question.
These evaluate on chains through pull-back, generalizing intersection.
Secondly, we need manifolds with corners to define a product using fiber product, as boundaries are needed to define cohomology and corners arise immediately when taking fiber products of manifolds with boundary. Even though, for example, the collection of smooth maps from simplices constitute a cochain complex additively, taking fiber product quickly leads to more general representing objects.

While there are a number of treatments of homology and cohomology that employ manifolds and their generalizations \cite{Whit47, BRS76, FeSj83, Krec10, Kahn01, Zing08, Joyc15}, we find {geometric cohomology}, first developed by Lipyanskiy in the preprint \cite{Lipy14}, to be the most suitable for connecting differential and combinatorial topology.
In \cite{FMS-foundations}, we have filled in details of this theory as well as equipping it with a (partially defined) multiplicative structure on cochains.
We now give an overview of geometric cohomology referring to \cite{FMS-foundations} for a more complete exposition.

\subsection{Manifolds with corners}
 
We follow a careful development by Joyce \cite{Joy12}.
Let $\R^n_k = [0,\infty)^k \times \R^{n-k} \subset \R^n$, and let $x_i \colon \R^n_k \to \R$ denote projection onto the $i$th coordinate.
Define a map between open subsets of these spaces to be smooth if it can be extended to a smooth map of the ambient Euclidean space in a neighborhood of each point.
We carry over the definitions of smooth charts and atlases as in the standard setting, and we choose to work with subspaces of $\R^\infty$ in order to have a set of such objects.
The following definitions are from \cite[Section 2]{Joy12}.

\begin{definition}
	A {\bf manifold with corners}, or simply a \textbf{c-manifold}, is a subspace of some $\R^N \subset \R^\infty$ that is a topological manifold with boundary together with an atlas of smooth local charts modeled on $\R^n_k$.
	
	The smooth real-valued functions on a manifold with corners $W$ are those $f$ such that for each chart $\phi \colon U \subset \R^n_k \to W$ the composition $f \circ \phi \colon U \to \R$ is smooth.	
	
	A map $f \colon W \to V$ from a manifold with corners to a manifold without boundary 
	is {\bf smooth} if the composition of $f$ with any smooth real-valued function on $V$ is 
	a smooth real-valued function on $W$.
	
	The {\bf tangent bundle} of a manifold with corners is the space of derivations of the ring of smooth real-valued functions.
\end{definition}

By modeling on $\R^n_k$, our category includes manifolds ($k=0$) and manifolds with boundary ($k=1$), as well as cubes and simplices, but not the octahedron, for example, as the cone on $[0,1] \times [0,1]$ is not smoothly modeled by any $\R^n_k$.

Joyce extends the notion of smooth map to maps between manifolds with corners by making fairly stringent requirements on points which map to the boundary of the codomain, so that, in particular, fiber products are well-behaved.
In Joyce's terminology, the extension of our definition of smooth maps into manifolds with corners is called {\bf weakly smooth}.
We will not require this more stringent definition of a smooth map between manifolds with corners, as the only time we will consider fiber products over manifolds with corners will be when those products are zero-dimensional, for which we give an ad hoc treatment.

We will use boundaries of manifolds with corners, which are defined through their natural stratifications.

\begin{definition}
	A point $w$ in a manifold with corners $W$ has \textbf{depth} $k$ if there is a chart from an open subset of $\R^n_k$ that sends the origin to $w$.
	Define the {\bf corner-strata} $S^k(W) \subseteq W$ to be the set of elements having depth~$k$.
\end{definition}

If $W$ is a manifold with boundary, then $S^0(W)$ is its interior and $S^1(W)$ is its boundary. 
But, if $W$ is a general manifold with corners, deeper corner strata need to be incorporated in the boundary.
Because of this, the boundary is naturally a c-manifold over $W$ (that is, a map from a c-manifold to $W$), rather than a subspace of $W$.
Again see \cite[Section 2]{Joy12} for further details.

\begin{definition}
	A \textbf{local boundary component} $\beta$ of $W$ at $x \in W$ is a consistent choice of connected component $b_U$ of $S^1(W) \cap U$ for any neighborhood $U$ of $x$, with consistent meaning that $b_{U \cap U'} \subset b_{U} \cap b_{U'}$.
\end{definition}

Since this notion is local, the number of such components is determined by depth.
Considering the origin in $\R^n_k$, for any $k \geq 0$, points having depth $k$ have exactly $k$ local components. 
For example, $S^1(\interval^3)$ consists of the interiors of two-dimensional faces, and any sufficiently small neighborhood of a corner intersects exactly three of these.

\begin{definition}
	Let $W$ be a manifold with corners. Define its {\bf boundary} $\partial W$ to be the space of pairs $(x, \beta)$ with $x \in W$ and $\beta$ a local boundary component of $W$ at $x$.
	Define $i_{\partial W} \colon \partial W \to W$ by sending $(x,\beta)$ to $x$. 
\end{definition}

The boundary $\partial W$ is itself a manifold with corners, and the boundary map $i_{\partial W}$ is an immersion.
If $W$ is oriented, we orient $\partial W$ by stipulating that an outward normal vector followed by an oriented basis of $\partial W$ yields an oriented basis for $W$.
Taking boundary satisfies the Leibniz rule.

For geometric cohomology, we need co-orientations rather than orientations.
These are treated below and are more involved, requiring care to develop in \cite{FMS-foundations}.

We let $\partial^k W$ denote $\partial (\partial^{k-1} W)$ with $\partial^0 W = W$, and we let $i_{\partial^k W}$, or simply $i_{\partial^k}$, denote the composite of the $i_{\partial^i W}$ maps sending $\partial^k W$ to $W$.

Recall the standard notion of transversality of two maps, defined locally by having the tangent space at an image point spanned by the images 
of tangent spaces of preimages.

\begin{definition}
	Let $f \colon V \to M$ and $g \colon W \to M$ be smooth maps from manifolds with corners to a manifold without boundary.
	We say $f$ and $g$ are \textbf{transverse}, denoted $f \pf g$, if $f|_{S^k(V)}$ and $g|_{S^\ell(W)}$ are transverse for all $k, \ell$. This is equivalent to requiring all pairs $fi_{\partial^kV}$ and $gi_{\partial^\ell W}$ be transverse in the standard sense.
	
	Suppressing maps from the notation, define the \textbf{pull-back} or \textbf{fiber product} $V \times_M W$ as the subspace of $(x, y) \in V \times W$ with $f(x) = g(y)$.
\end{definition}

We will use the term pull-back when we want to emphasize its map to $V$ or $W$, while the fiber product is to be considered over $M$.
The following analysis of fiber products is standard -- see for example Proposition~7.2.7 of \cite{MaDo92}. 

\begin{theorem} \label{pullback}
	Let $f \colon V \to M$ and $g \colon W \to M$ be smooth maps from manifolds with corners to a manifold without boundary.
	If $f \pf g$ then the fiber product $V \times_M W$ is a manifold with corners with 
	\begin{equation*}
	S^i(V \times_M W) = \bigsqcup_{k + \ell = i} S^k(V) \times_M S^\ell(W).
	\end{equation*}
	Moreover, the maps from the fiber product to $V$, $W$, and $M$ are weakly smooth.
\end{theorem}

To generalize this theorem when $M$ is also a manifold with corners requires substantial additional hypotheses in the definition of transverse smooth maps.
Such a generalization is a central result in \cite{Joy12}. We only require this case and the case of manifolds of complementary dimension, discussed below.

\subsection{Geometric cohomology}

Geometric cohomology is a cohomology theory for smooth manifolds defined via proper co-oriented maps from manifolds with corners.
It agrees with singular cohomology, but with different representatives at the cochain level it gives geometric approaches to both theory and calculations.
It is thus akin to de Rham theory in being defined through smooth manifold structure rather than continuous maps.
But, unlike de Rham theory, it is defined over the integers.

Geometric homology and cohomology were defined and developed in a preprint of Lipyanskiy \cite{Lipy14}.
But this preprint does not develop a multiplicative structure at the cochain level.
Moreover, while Lipyanskiy shows geometric homology groups are isomorphic to singular homology, he does not state or prove the corresponding fact for cohomology.
We give a full treatment addressing these points and others in \cite{FMS-foundations}, reviewing here only what we need to compare geometric and cubical cohomology. 

We first define co-orientations.
Recall that one definition of an orientation of a bundle is an equivalence class, up to positive scalar multiplication, of an everywhere non-zero section of the top exterior power of the bundle.

\begin{definition} \label{D: co-orientations}
	Let $E \to M$ be a rank $d$ vector bundle.
	If $d > 0$, define $\Or(E)$ to be $\bigwedge^d E$, and if $d = 0$, define $\Or(E)$ to be the trivial rank one bundle.
	
	A \textbf{co-orientation} of $g \colon W \to M$ is an equivalence class, up to positive scalar multiplication, of an everywhere non-zero section of $\Hom \big( \Or(TW),\, \Or(g^*TM) \big)$.
	We say that $g$ is \textbf{co-orientable} if a co-orientation exists.
\end{definition}

The local triviality of the determinant line bundle of a manifold means being able to choose a consistent basis vector over sufficiently small neighborhoods. 
We call such a choice of basis vectors around a point, which we typically do not specify, a \textbf{local orientation}, and often denote the local orientation for $W$ by $\beta_W$.
We use ordered-pair notation for co-orientation homomorphisms, with $\omega = (\beta_W, \beta_M)$ being the co-orientation that sends the local orientation $\beta_W$ at $x\in W$ to a local orientation $\beta_M$ for $g^*(TM)_x\cong T_{g(x)}M$. 	

We can equivalently define a co-orientation as a choice of isomorphism $\Or(TW)\cong \Or(g^*TM)$, again up to positive scalar multiplication.
Thus, if $f$ is co-orientable and $W$ is connected, any local co-orientation uniquely extends to a global co-orientation.
If a map is co-orientable, it has exactly two co-orientations, which we say are \textbf{opposite}, with, for example, the opposite of $\omega$ above being $(\beta_W, -\beta_M)$, which we also write as $-(\beta_W, \beta_M)$. 

Co-oriented maps compose in an immediate way, forming a category. 
Namely, given $V \xr{f} W \xr{g} M$ and co-orientations $\Or(TV) \to \Or(f^*TW)$ and $\Or(TW)\to \Or(g^*TM)$, we simply compose the latter with the pullback of the former via $g^*$.

Like orientations, co-orientations are ``additional data.''
An exception is the {\bf intrinsic co-orientation of a diffeomorphism}, as the differential
in this special case induces a map on determinant line bundles.
A key case of co-orientation is the following.

\begin{definition} \label{D: normal co-or}
	Let $g \colon W \to M$ be an immersion with an oriented normal bundle $\nu$, with local orientation denoted by $\beta_\nu$.
	Define the \textbf{normal co-orientation} associated to $\beta_\nu$ locally as $\omega_{\nu} = (\beta_W, \beta_W \wedge \beta_\nu)$, where	$\beta_W$ is any choice of a local orientation of $W$.
	
	Conversely, if $g \colon W \to M$ is a co-oriented immersion, define the induced orientation of the normal bundle as the one whose normal co-orientation	agrees with the given one.
\end{definition}

\begin{definition} \label{V: maps are co-oriented}
	A \textbf{c-manifold over a manifold with corners} $N$ is a 
	manifold with corners $W$ with a weakly smooth, proper, co-oriented map $r_W \colon W \to N$, called the \textbf{reference map}.
	Two such are equivalent if there is a diffeomorphism $f \colon W \to W$ so that $r_W \circ f = r'_W$ and the composite of the co-orientation of $r_W$ with the intrinsic co-orientation of $f$ yields the co-orientation of $r'_W$.
	Let $\cman(N)$ denote the set of proper co-oriented c-manifolds over $N$.

	For an element of $\cman(N)$, we write $|W|$ for the \textbf{codimension} $|W|=\dim(N)-\dim(W)$. 
	Let $\cman^*(N)$ be the free abelian group generated by $\cman(N)$, graded by the \textbf{codimension} $|W|$, modulo the following relations:
	\begin{enumerate}
	\item ${V \sqcup W} = V + W$,
	\item ${{W}^{op}} = -W$, where ${W}^{op}$ denotes the co-oriented manifold over $M$ obtained by reversing the co-orientation.
	\end{enumerate}
\end{definition}

We take a free abelian group and then quotient by the first relation instead of defining sum as disjoint union as in \cite{Lipy14} since we define our manifolds with corners as subspaces of a fixed universe, which complicates self addition through union.
By these relations any element of $\cman^*(N)$ is represented by a single map from a likely disconnected manifold with corners, as in particular one can find as many ``copies'' as one needs of any manifold with corners embedded in $\R^\infty$. 

We freely and almost always abuse notation by using the domain $W$ to refer to the manifold over $N$, not $r_W$ or some other symbol, letting context determine whether we are referring to the entire data or the domain.
Our favorite class of c-manifolds over $N$ are submanifolds, for which this abuse is minor.

When $M$ has no boundary, we use $\cman(M)$ to construct a chain complex based on these objects that will compute cohomology.
To do so, we consistently co-orient boundaries, using composition of co-orientations.

\begin{definition} \label{D: boundary co-orientation}
	The {\bf standard co-orientation of a boundary inclusion} $\partial W \hookrightarrow W$ is the normal co-orientation associated to the outward-pointing orientation of $\nu_{\partial W \subset W}$.
	
	If $g \colon W \to M$ is co-oriented, the {\bf induced co-orientation} of $g|_{\partial W}$ is the composition of the standard co-orientation of $\partial W$ into $W$ with the pullback of the co-orientation of $g \colon W \to M$. 
\end{definition}

Our cochains will be equivalence classes of co-oriented c-manifolds over $M$, under an equivalence relation we define using the following concepts, which are taken from or inspired by the definitions of \cite{Lipy14}.

\begin{definition}\label{D: cman types}
	Let $V, W\in \cman(M)$  with reference maps $r_V$ and $r_W$.
	We say
	\begin{itemize}
		\item $V$ and $W$ are \textbf{equivalent} if there is a co-orientation preserving diffeomorphism $\phi \colon W \to V$ such that $r_V \circ \phi = r_W$.
		\item $W$ is \textbf{trivial} if there is a diffeomorphism $\rho \colon W \to W$ such that $r_W \circ \rho = r_W$
		and the composite of the co-orientation of $r_W$ with the co-orientation given by $D\rho$ is the opposite of the co-orientation of $r_W$.
		\item $W$ has \textbf{small rank} if the differential $D r_W$ is less than full rank at all points of $W$.
		\item $W$ is \textbf{degenerate} if it has small rank and ${\partial W}$ is the union of a trivial co-oriented 
		c-manifold over $M$ and one with small rank.
	\end{itemize}
\end{definition}

Rather than small rank, Lipyanskiy uses a condition called small image.
In our notation, $r_W$ has small image if there is an $r_T \colon T \to M$ with $T$ of smaller dimension than $W$ such that $r_W(W) \subseteq r_T(T)$.
The small rank condition is thus weaker, and we find it more manageable for purposes of defining a product while still providing a theory that is isomorphic to singular cohomology \cite{FMS-foundations}. 
 
A key example is the interval mapping to a point, which has small rank, but its boundary does not have small rank. 
It is nonetheless degenerate because its boundary is trivial.

\begin{definition} \label{D: geometric cohomology}
	Let $M$ be a manifold without boundary. Let $Q^*(M)$ denote the subgroup 
	of $\cman^*(M)$ generated by those equivalence classes that are either trivial or degenerate.
	We define the \textbf{geometric cochains} of $M$, denoted $C_\Gamma^*(M)$, as the quotient $\cman^*(M)/ Q^*(M)$.
	We denote the equivalence class of $W$ modulo $Q^*(M)$ by $\uW$. 
\end{definition}

The definitions are arranged so that geometric cochains form a chain complex.

\begin{proposition}
	If $V \in Q^*(M)$ then $\partial V \in Q^*(M)$.
	Moreover, for any $W \in \cman^*(M)$, $\partial^2 W \in Q^*(M)$.
\end{proposition}

Details can be found in \cite{Lipy14, FMS-foundations}.
Briefly, $\partial^2 W$ always has a $C_2$-action, permuting the local boundary components attached to points in $S^2(W)$.
Moreover, under our co-orientation conventions, the two vectors appended to the co-orientation of $W$ to obtain one for $\partial^2 W$ over the same point in $S^2(W)$ differ by a transposition, so this $C_2$-action is co-orientation reversing.
This fact about $\partial^2$ not only eventually shows that $d_\Gamma^{\,2} = 0$ but is first needed to show that the boundary of a degenerate map is degenerate. 

\begin{definition}
	Define a differential $d_\Gamma \colon C_\Gamma^*(M) \to C_\Gamma^{*+1}(M)$ by sending $\uW$ to $ \underline{\partial W}$, making $C_\Gamma^*(M)$ into a chain complex called the {\bf geometric cochain complex}.
	We denote its homology by $H^*_\Gamma(M)$, the \textbf{geometric cohomology} of $M$. 
\end{definition}

We focus on the case in which $M$ is a manifold without boundary primarily because 
the theory with boundary requires relative constructions.
For example, the identity $\R\to\R$ generates $H_\Gamma^0(\R)\cong \Z$, but the identity $[0,1] \to [0,1]$ is not a cocycle unless we quotient out by mappings to the boundary.
A definition for manifolds with boundary, or more generally corners, would also require boundary restrictions for transversality of weakly smooth maps, as developed for example in \cite{Joy12}.
We leave such generalizations to further work.

Lipyanskiy also develops a theory of geometric chains, as opposed to cochains, using compact domains and orientations.
In Section 6 of \cite{Lipy14}, he shows that the homology theory based on geometric chains satisfies some of the Eilenberg-Steenrod axioms, which is is enough to deduce in Section 10 that geometric homology is isomorphic to singular homology.
Lipyanskiy does not treat geometric cohomology in the same detail, and in particular he does not claim that it is isomorphic to singular cohomology.
We prove this is true in \cite{FMS-foundations}, but the proof requires the development of additional tools -- either cubulations (or triangulations) and the 
intersection homomorphism as in Definition~\ref{D: intersection homomorphism} below, or by using work of Kreck and Singhof \cite{Krec10, Krec10b}.

\begin{theorem} \label{T: geometric singular isomorphism}
	On the category of smooth manifolds (without boundary) and continuous maps, geometric cohomology is isomorphic to singular cohomology with integer coefficients.
	That is, the functors $H^*_\Gamma$ and $H^*(\cdot \,;\Z)$ are naturally isomorphic.
\end{theorem}

The functoriality here at the cohomology level is fully defined with respect to all continuous maps.
Given an element of cohomology represented by $r_W \colon W \to M$, choose a smooth map in the homotopy class of $f$ that is transverse to $r_W$ and then pull back. 
It is shown in \cite{FMS-foundations} that this process gives a well-defined induced map $f^*$.

There is no full functoriality at the cochain level, since an $f$ cannot be transverse to all c-manifolds over $M$.
But there is a quasi-isomorphic subcomplex of $C_\Gamma^*(M)$ consisting of cochains that are transverse to $f$ which can be pulled back.
This is analogous to having only a partially-defined fiber product, as we introduce in Section~\ref{S: fiber product section}.
Such ``partial functoriality'' of cochains will be needed for applications of the main results of this paper at the cochain level.

\subsection{Cubical structures and intersections}

We now bring together the two structures we have been developing, geometric cochains and cubulations.
We construct a quasi-isomorphism from (a sub-complex of) geometric cochains to cubical cochains, and in subsequent sections we will exhibit a vector field flow that binds these structures multiplicatively.
In the de Rham setting, integration provides a relationship between differential forms and cochains.
For geometric cochains the intersection homomorphism plays a similar role.

A smooth cubulation is one for which characteristic maps are smooth maps of manifolds with corners.
Smooth cubulations exist for any smooth manifold, as in the following construction of \cite{ShSh92}.
Start with a smooth triangulation (see for example \cite[Theorem 10.6]{MUNK66} for the existence of such). 
Consider the cell complex that is dual to its barycentric subdivision. Intersecting those dual cells with each simplex in the triangulation provides a subdivision of the simplex into cells that are linearly isomorphic to cubes.
Moreover, starting with an ordered triangulation -- obtained for example by taking a barycentric subdivision -- such a cubical decomposition embeds cellularly into the cubical lattice of $\R^\infty$, and thus it is the geometric realization of a cubical complex.

Since cubes are oriented manifolds with corners, the cubical chain complex maps injectively to Lipyanskiy's geometric chain complex.
But in contrast with the evaluation of singular cochains on singular chains, which is purely algebraic, the natural evaluation of geometric cochains on geometric chains is defined through intersection or, more generally, pull-back.

\begin{definition}
	Let $M$ be a manifold without boundary equipped with a smooth cubulation $|X| \to M$.
	We say that $W \in \cman(M)$ is \textbf{transverse} to $X$ if its reference map is transverse to each characteristic map of the cubulation.

	We denote by $\cman_{\pf X}^*(M)$ the subcomplex of $\cman(M)$ generated by maps transverse to $X$ and by $C^*_{\Gamma \pf X}(M)$ the corresponding quotient by its intersection with $Q^*(M)$.
\end{definition}

We will not reference $X$ when it is clear from the context. 
The subsets $\cman_\pf^*(M)$ and $C^*_{\Gamma \pf}(M)$ are well-defined chain complexes since transversality of the maps representing geometric cochains by definition includes transversality of their restrictions to all strata, in particular their boundaries.

A key technical result, whose proof is given in \cite{FMS-foundations}, 
is that these transversality conditions do not change cohomology.

\begin{theorem} \label{T: transverse complex}
	For any cubulated manifold $M$, the inclusion $C^*_{\Gamma \pf}(M) \to C^*_\Gamma(M)$ is a quasi-isomorphism.
\end{theorem}

We next obtain cubical cochains from elements in $C^*_{\Gamma \pf}(M)$ essentially by counting intersections.
We will require reference to various components of the intersection homomorphism, so we set them aside in a series of closely related definitions.

\begin{definition}
	A \textbf{signed set} is a finite set $S$ with a \textbf{sign function} $\mathrm{sgn} \colon S \to \{\pm 1\} \subseteq \Z$.
	The \textbf{signed cardinality} of such a set is $\sum_{p \in S} \mathrm{sgn}(p)$, which we denote by $\alpha(S)$.
\end{definition}

The signed sets we count are discrete intersections -- or more generally pull-backs -- of manifolds with corners.

\begin{definition}
	We say that c-manifolds over a c-manifold $N$ are \textbf{complementary} if 

\begin{itemize}
\item their codimensions -- or equivalently dimensions --
	add to the dimension of $N$,

\item  over any $S^i(N)$ with $i>0$ their images are disjoint, and

\item over $S^0(N)$ they are transverse in the usual sense.
\end{itemize}
\end{definition}

The disjointedness condition over strata is also a transversality condition, which can be viewed as a special case of a full notion of transversality over a manifold with corners as in \cite{Joy12}.
By focusing on complementary manifolds, the stringent boundary conditions for general transversality reduce to an expected disjointedness over all but the interior.
 
If $W$ is a c-manifold over $M$ that is transverse to a cubulation and $E$ is a cube of complementary dimension, then the intersection of the image of $W$ with $E$ is discrete.
Furthermore, the pull-back $W \times_M E$ is finite since $r_W$ is proper.
About any point $p \in r_W^{-1}(E) \subset W$, the reference map $r_W$ is locally an embedding, and thus locally has a normal bundle.
As noted above, the co-orientation of $W$ determines an orientation of the normal bundle, 
which at the intersection point $r_W(p)$ can be identified with its tangent space in $E$.
Thus this orientation of the normal bundle can be compared with the standard ordering of basis vectors of $E$ at $r_W(p)$, when identified with a standard cube via its characteristic map.
(This local orientation of $E$ is not immediately related to an orientation of $M$.)

\begin{definition} \label{D: intersection number}
	Let $W, E \in \cman({N})$ be complementary, and let $W$ be co-oriented while $E$ is oriented. Define $\mathrm{Int}_N(W, E)$ to be the signed set given 
	by the pull-back, with sign function given by comparing the normal co-orientation of $W$ with the orientation of $E$.
	Define the intersection number $I_N(W,E)$ to be $\alpha(\rm{Int}_N(W, E))$.
\end{definition}

\begin{definition} \label{D: intersection homomorphism}
	Let $M$ be a manifold without boundary with a cubulation $|X| \to M$.
	The \textbf{intersection homomorphism}
	\begin{equation*}
	\cI \colon \cman^*_{\Gamma \pf X}(M) \to \cochains(X)
	\end{equation*}
	is the grading preserving linear map defined by sending $W$ to the cochain whose value on $E \in X$ is $\cI(W)(E) = I_M(W, E)$.
\end{definition}

The intersection of a cube with an element of $\cman(M)$ that is trivial, as in Definition \ref{D: cman types}, will give a canceling count, and there can be no intersections with small rank reference maps.
Therefore, the intersection homomorphism vanishes on $Q^*(M)$, and there is an induced a map on geometric cochains.
We show in \cite{FMS-foundations} that this is a chain map. 
The proof  is akin to the proof that degree of a smooth map is homotopy invariant, through the classification of compact one-manifolds.
Indeed, on some $(n+1)$-cube $E$ both $\delta \cI(\uW)$ and $\cI(\underline{\partial W})$ are counts of $0$-manifolds over $E$, which together are boundaries of the pull-back of $W$ and $E$, a $1$-manifold.

We refer to this induced map as the \textbf{intersection chain map} and denote it, abusively, also by $\cI \colon C_\Gamma^*(M)\to C^*(X)$.

\begin{theorem} \label{T: stokes}
	The map $C^*_{\Gamma \pf X}(M) \to \cochains(X)$ induced by the intersection homomorphism is a surjective quasi-isomorphism.
\end{theorem}

Surjectivity is clear since we can find for any cube a small submanifold transversally passing through it at one point. 
The quasi-isomorphism result is proven in \cite{FMS-foundations}.

\subsection{Fiber product} \label{S: fiber product section}

We now endow geometric cochains with a (graded) commutative product given by intersection of immersed submanifolds, or fiber product more generally.
This product is partially defined, as it must be if it is to be commutative and induce the cup product in cohomology.
The construction ends up being delicate since our cochains are themselves equivalence classes.
Indeed, Lipyanskiy only discusses multiplicative structure at the level of cohomology in Section~5 of \cite{Lipy14}.
Joyce's M-cohomology \cite{Joyc15}, which has more complicated representing objects, is also endowed with cochain-level product structure,
 after considerable effort.

We start at the level of c-manifolds over $M$, a manifold with no boundary.
Even here, substantial care in \cite{FMS-foundations} is taken to define a co-orientation
on the fiber product of co-oriented maps. We summarize the results as follows, recalling that $|V|$ stands for the codimension $\dim(M)-\dim(V)$. 

\begin{theorem} \label{T: pull-back co-or}
	Let $V$ and $W$ be transverse c-manifolds over $M$, a manifold without boundary,
	with co-orientations $\omega_V$ and $\omega_W$.
	There is a unique co-orientation $\omega_P$ of $P = V \times_M W$, which depends on $\omega_V$ and $\omega_W$, with the following properties:
	\begin{enumerate}
	\item Reversing either $\omega_V$ or $\omega_W$ results in reversing $\omega_P$.
	\item The co-orientations of $V \times_M W$ and $W \times_M V$, when compared by composition with the restriction of the diffeomorphism which sends 
	$V \times W$ to $W \times V$, differ by the sign $(-1)^{|V||W|}$.
	\item $\partial ( V \times_M W) = (\partial V) \times_M W \sqcup (-1)^{|V|} V \times_M (\partial W)$.
	\item Let $r_V$ and $r_W$ be immersions with oriented normal bundles, so $V \times_M W$ itself is an immersion whose
	normal bundle $\nu$ is isomorphic to the direct sum of the normal bundles of $V$ and $W$. 
	Orient $\nu$ by the orientation of the normal bundle of $V$ followed by that of $W$.
	Then the co-orientation $\omega_P$ agrees with $\omega_{\nu}$.
	\end{enumerate}
\end{theorem}

We call the co-orientation $\omega_P$ the {\bf product co-orientation}.

\begin{definition}
	If $V, W \in \cman^*(M)$ are transverse, define $V \bullet_M W$ to be $V \times_M W$ with the product co-orientation.
\end{definition}

By Theorem~\ref{T: pull-back co-or}, $ \cman^*(M)$ is thus a partially-defined graded commutative ring.
We next address the passage to cochains, giving a partially defined differential graded commutative algebra.

\begin{definition} \label{D: cochain trans}
	We say that two geometric cochains $\uV, \uW \in C^*_{\Gamma}(M)$ are \textbf{transverse} if they possess representative elements in $\cman^*(M)$ of the form 
	$V = V_T + V_Q$ and $W = W_T + W_Q$
	such that:
	\begin{enumerate}
	\item ${V_T}$ and ${W_T}$ are transverse, and
	\item ${V_Q}$ and ${W_Q}$ are elements of $Q^\ast(M)$.
	\end{enumerate}
	
	With such decompositions fixed, we define the fiber product $\uV \bullet_M \uW \in C^*_\Gamma(M)$ 
	to be the geometric cochain represented by $V_T \bullet_M W_T$.
\end{definition}

A delicate argument in \cite{FMS-foundations} gives the following.

\begin{theorem} \label{P: product}
	The fiber product $\bullet_M$ descends to a well-defined, though only partially-defined, product on $C_\Gamma^*(M)$, which in turn passes to a fully-defined product on $H_\Gamma^*(M)$. Under the isomorphism of Theorem~\ref{T: geometric singular isomorphism}, the induced product on geometric cohomology agrees with cup product on singular cohomology.
\end{theorem}

The agreement of fiber product with cup product at the cohomology level will also follow from our main cochain-level result, Theorem~\ref{T: main theorem}. 
But there are more direct elementary arguments, given in \cite{FMS-foundations}, that prove agreement at the cohomology level without yielding any insight at the cochain level.



\section{Logistic vector field} \label{S: logistic}

We construct a vector field associated to a cubulation on a manifold that, in a sense, gives a smooth extension of the cubical poset structure.
This vector field is a cousin of the standard vector field on a triangulated manifold whose zeros coincide with barycenters of the triangulation.
They can both, for example, be used to prove the Poincar\'e-Hopf theorem, equating the signed count of zeros of a generic vector field and the Euler characteristic of a manifold.

For our applications, as in the Poincar\'e-Hopf theorem, the most significant aspect of the considered vector field is its zero locus.
In the cubical context, one can naturally take products of vector fields, and in particular define a family of vector fields on cubes starting with any vector field on the interval.  
Such families are compatible across face structures if the vector field is zero on the boundary of the interval.
We require a vector field on the interval whose only zeros are at the boundary, in which case the vector fields on cubes will only have zeros at their corners.  
We choose to start from arguably the simplest such vector field, which is amenable to explicit analysis and whose dynamics have been extensively studied.

\begin{definition}
	The \textbf{logistic vector field} $\f^n$ over $\R^n$ is defined by
	\begin{align*} 
	\f^n(x) & = \sum_{i=1 }^n x_{i} (1 - x_{i}) \e_{i}
	\end{align*}
	where $\e_i = \frac{\partial\ }{\partial x_i}$.
	We denote the time $t$ flow of $x$ along it by $\f_t^n(x)$.
\end{definition}

\subsection{Naturality}

We will exclusively consider the restriction of $\f^n$ to the unit cube $\interval^n$, where we have the following key compatibility, which allows us to extend this vector field to any cubulated manifold.

\begin{lemma} \label{L: f is natural}
	The vector fields $\f^n$ are natural with respect to face inclusions. That is, for integers $k$ and $n$ with $1 \leq k \leq n$,
	\begin{equation*}
	\f^n|_{\delta_{k}^\varepsilon(\interval^{n-1})}=(\delta_{k}^\varepsilon)_*(\f^{n-1}).
	\end{equation*}
\end{lemma}

\begin{proof}
	We compute 
	\begin{align*}
	(\delta_{k}^\varepsilon)_*(\f^{n-1}_x) & = (\delta_{k}^\varepsilon)_* \left(\sum_{i=1}^{n-1 } x_{i} (1 - x_{i}) \e_{i} \right) \\ & = 
	\sum_{i=1 }^{k-1} x_{i} (1 - x_{i}) \e_{i}\ +\
	\sum_{i=k}^{n-1} x_{i+1} (1 - x_{i+1}) \e_{i+1} \\ & =
	\sum_{i=1}^{k-1} x_{i} (1 - x_{i}) \e_{i}\ +\
	0 \cdot \e_k\, + \
	\sum_{i=k+1}^n x_{i} (1 - x_{i}) \e_{i} \\ & =
	\f^n|_{\delta_k^\varepsilon(x)}
	\end{align*}
	as claimed.
\end{proof}

Given these face compatibilities, we will typically leave off the dimension index and simply write $\f$, allowing context to determine whether we consider $\f^n$ on $\interval^n$ or $\f^k$, $k<n$, on one of the $k$-faces of $\interval^n$.

\begin{definition}
Let $ |X| \to M$ be a smooth cubulation of a manifold $M$ with $\iota_c : \interval^d \to M$  the characteristic map of a cube $c$.
 Define the logistic vector field on $M$ associated to this cubulation by, for each $c$, applying the derivative
of $\iota_c$ to the logistic vector field on $\interval^d$.  
\end{definition}

By the fact that characteristic maps are each homeomorphisms onto their  images, Lemma ~\ref{L: f is natural} implies the logistic vector field on $M$ associated to $|X| \to M$ is well defined.  

We can now formulate precisely a sense in which the logistic vector field gives an extension of the vertex ordering of the cubical set $X$ that cubulates $M$.
In general, the flow of a vector field at a point in a manifold  defines a flow-line map $(a,b) \subset \R \to M$ or in some cases $[-\infty, \infty] \to M$.  
The poset structure on $(a,b)$ or 
$[-\infty, \infty]$ can then be imposed on the flow-line, well defined as it is independent of point at which the flow-line is centered.
In some cases these flow-line posets extend to a poset structure on all of $M$.  
The logistic flow is one such case, and the resulting poset structure when restricted to vertices agrees with  the poset structure, 
 which is part of the cubical structure as in 
Definition~\ref{D:cubical}.  Indeed, as will be immediate from our next discussion,
flow-lines all extend to $[-\infty, \infty]$, starting at some vertex $v$ and ending at a vertex $w$ with $v \leq w$ in the cubical ordering. 

\subsection{Logistic flow}

We next explicitly describe the logistic flow diffeomorphism.

\begin{lemma} \label{L: bernoulli equation}
	The {\bf logistic flow} $\f_t(x) = (x_1(t), \dots, x_n(t))$ exists for any $t \in \R$ and any $x \in \interval^n$,
	and is explicitly given by the logistic function
	\begin{equation*}
	x_i(t) = \frac{x_i(0) \; e^t}{x_i(0)(e^t-1)+1.}
	\end{equation*}
	In particular, $x_i(t) = \varepsilon$ if $x_i(0) = \varepsilon$ for $\varepsilon \in \{0, 1\}$.
\end{lemma}

The proof is to check that the $x_i(t)$ given satisfy $\frac{d x_i}{dt} = x_i(1- x_i)$ -- indeed, these $x_i$ are standard and can be found using separation of variables -- and then appeal to existence and uniqueness of single variable ordinary differential equations.

The inverse function associated to this flow is also elementary. In a single variable, if $y = \f_t(x)$ then solving the flow as expressed in Lemma \ref{L: bernoulli equation} in terms of $x$ gives us 
\begin{equation} \label{E: inverse flow}
x = \frac{y}{e^t - y(e^t-1)}.
\end{equation}

Since 
\begin{equation*}
\lim_{t \to + \infty} \frac{e^t}{e^t + c} = 1
\quad \text{ and } \quad
\lim_{t \to - \infty} \frac{e^t}{e^t + c} = 0,
\end{equation*}
for $c\neq 0$, we have the following. 

\begin{corollary} \label{C: limit of points along the flow}
	For every $x \in \interval^n$, the limits $x^\pm = \lim_{t \to \pm \infty} \f_t(x)$ 	exist, and we have
	\begin{equation*}
	x^+_i = \begin{cases}
	0, & \text{ if } x_i(0) = 0, \\ 1, & \text{ otherwise,}
	\end{cases}
	\quad \text{ and } \quad
	x^-_i = \begin{cases}
	1, & \text{ if } x_i(0) = 1, \\ 0, & \text{ otherwise.}
	\end{cases}
	\end{equation*}
\end{corollary}

By taking derivatives of the formulas in Lemma \ref{L: bernoulli equation} with respect to the $x_i(0)$, treated as coordinates, we immediately identify the flow diffeomorphism on tangent spaces.

\begin{corollary} \label{C: pushforward of vectors along the flow}
	The Jacobian matrix representing the differential of the diffeomorphism $\f_t:\interval^n \to \interval^n$ for a fixed time $t$	is diagonal with entries
	\begin{equation*}
	\left(D_x \f_t\right)_{i,i} = \frac{e^t}{(x_i(e^t -1) + 1)^2.}
	\end{equation*}
\end{corollary}

\subsection{Neighborhoods} \label{S: neighborhoods}

Given a face $F$ of $\interval^n$, we will focus on two families of subsets of $F$ that are parameterized by real numbers $u \in (0,1)$, the \textbf{lower} and \textbf{upper subsets} of $F$
\begin{align*}
L_u(F) & = \{x \in F\ |\ \forall j \in F_{01},\ x_j \leq u\}, \\
U_u(F) & = \{x \in F\ |\ \forall j \in F_{01},\ x_j \geq u\}.
\end{align*}
We also need sets that are neighborhoods of these in the bound directions.
Let $N_\epsilon L_u(F)$ and $N_\epsilon U_u(F)$ consist of those points whose free variables are constrained as above and whose bound variables are within $\epsilon$ of those of $F$.
For example, the first set can be described explicitly as
\begin{align*}
N_\epsilon L_u (F) & = \left\{(x_1,\ldots, x_n) \in \interval^n\ \middle| 
\begin{array}{ll}
\forall j \in F_0, & x_j < \epsilon, \\
\forall j \in F_{01}, & x_j \leq u, \\
\forall j \in F_1, & x_j>1-\epsilon,
\end{array}\right\},
\end{align*}
and the second can be described analogously.
We will refer to $N_\epsilon L_u(F)$ and $N_\epsilon U_u(F)$ respectively as a \textbf{lower} and \textbf{upper neighborhood} of $F$ despite $F$ not being a subset of either.
Please compare with Figure~\ref{F: lower subspace and nbhd}.

Consistent with this, the $L^\infty$ neighborhood $N_\epsilon(F)$ consists of those points whose bound variables are within $\epsilon$ of those of $F$.
In particular, if $F$ is a terminal face then $N_\epsilon(F)=\{ (x_1,\ldots, x_n)\mid x_j>1-\epsilon,\ j\in F_1\}$, and if $F$ is initial then $N_\epsilon(F)=\{ (x_1,\ldots, x_n)\mid x_j<\epsilon,\ j\in F_0\}$.

\begin{figure}[!h]
	\centering
	\begin{subfigure}{.32\textwidth}
		\includegraphics[scale=.7]{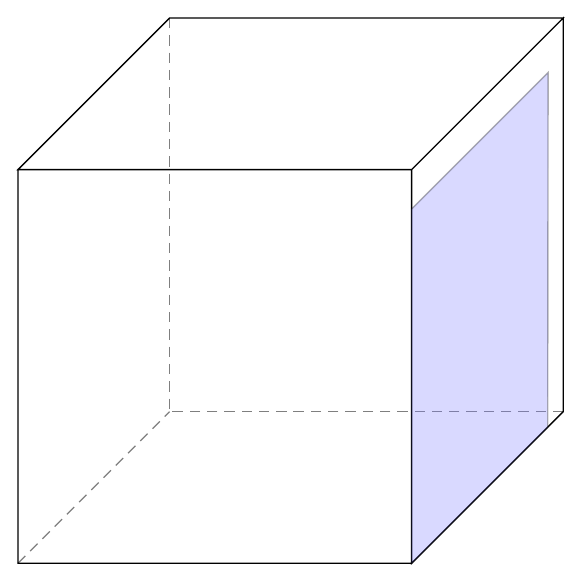}
		\hfill
	\end{subfigure}
	\begin{subfigure}{.32\textwidth}
		\hfill
		\includegraphics[scale=.7]{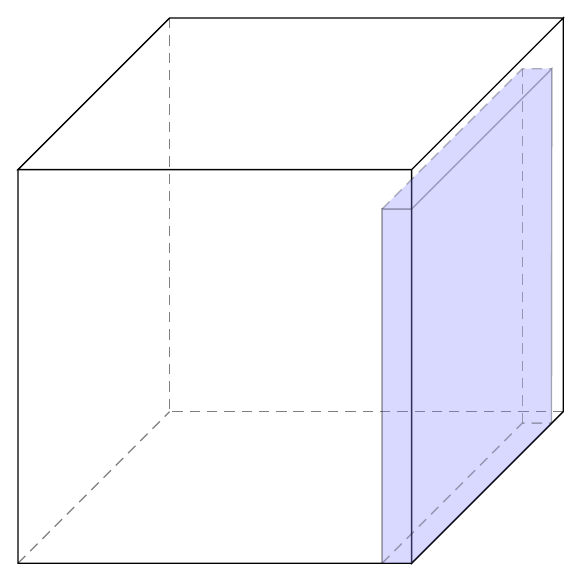}
	\end{subfigure}
	\caption{A lower subset $L_u(F)$ and neighborhood $N_\epsilon L_u(F)$ of $F = (\emptyset, \{2,3\}, \{1\})$.}
	\label{F: lower subspace and nbhd}
\end{figure}

Just as flow lines go between vertices of the cubulation, respecting their order, we now show that the flow takes a lower neighborhood of
 one face $F$ to a neighborhood of the ``next'' face $F^+$, 
whose initial vertex is the terminal vertex of $F$.

\begin{lemma} \label{L: flow to initial and terminal faces}
	Let $F$ be a face of $\interval^n$. For any $\epsilon > 0$ and $u,r \in (0, 1)$ we have
	\begin{align*}
	\f_{t}(N_r U_u (F)) &\subseteq N_\epsilon (F^+), \\
	\f_{-t}(N_r L_u (F)) &\subseteq N_\epsilon (F^-),
	\end{align*}
	for all $t$ sufficiently large.
\end{lemma}

\begin{proof}
	Recall that if $F = (F_0, F_{01}, F_1)$ then $F^+ = (\emptyset, F_0, F_{01} \cup F_1)$, so the set $N_\epsilon (F^+)$ consists of points $x \in \interval^n$ whose $j\th$ coordinate for $j \in F_{01} \cup F_1$ is greater than $1-\epsilon$.
	Using that all elements in $N_r U_u (F)$ have $j\th$ coordinates for $j \in F_{01} \cup F_1$ greater than or equal to $u > 0$ (i.e.\ these coordinates are bounded away from $0$), Corollary~\ref{C: limit of points along the flow} together with the observation that the flow is order preserving finishes the proof when $t \to \infty$.
	The case when $t \to -\infty$ is proven similarly.
\end{proof}

The next lemma says something about the ``aspect ratios'' of tangent spaces under the flow $\f_t$, comparing the amount of stretching/compressing in directions orthogonal to an initial or terminal face to the amount of stretching/compressing parallel to the face as we move from $T_x\interval^n$ to $T_{y}\interval^n$ with $y=\f_t(x)$. In particular, if we flow long enough and stay close to appropriate upper or lower neighborhoods, we can make these ratios arbitrarily small. 

\begin{lemma} \label{L: jacobian ratios}
	Let $F$ be a face in $\interval^n$. For any $\varepsilon > 0$ and $u \in (0,1)$ there exist $\delta > 0$ and $T\in \R$ such that for any $y \in N_\delta L_u(F^+)$ and $t>T$ we have the bound on ratios
	\begin{equation} \label{E: jacobian ratios}
	\frac{\left|\left(D_{x}\f_t\right)_{i,i}\right|}
	{\left|\left(D_{x}\f_t\right)_{j,j}\right|} < \varepsilon,
	\end{equation}
	where $x$ is such that $y = f_t(x)$, $i \not\in F^+_{01}$, and $j \in F^+_{01}$. Similarly, $\delta > 0$ and $T\in \R$ can be chosen such that \eqref{E: jacobian ratios} holds for any $y \in N_\delta U_u(F^-)$ and $t>T$ with $y = f_{-t}(x)$, $i \not\in F^-_{01}$, and $j \in F^-_{01}$.
\end{lemma}
In other words, given an $\epsilon$, there is a neighborhood $N_\delta L_u(F^+)$ and a time $T$ so that for all greater times every point that has flowed into $N_\delta L_u(F^+)$ has its tangent space sufficiently ``squashed'' by the flow that the ratios in the normal directions to $F^+$ versus the tangential directions are all smaller than $\epsilon$. 

\begin{proof}
	We only prove the forward flow case since the backward flow case is analogous.
	Recall from Equation~\eqref{E: inverse flow} that if $y = \f_t(x)$ then $x = \frac{y}{e^t-y(e^t-1)}$. 
	Plugging this into the Jacobian formula of Corollary \ref{C: pushforward of vectors along the flow} gives diagonals of the Jacobian of the form
	\begin{equation*}
	\frac{e^t}{\left(\frac{y}{e^t-y(e^t-1)}(e^t -1) + 1\right)^2} \ = \
	\frac{e^t}{\left(\frac{e^t}{e^t-y(e^t-1)}\right)^2} \ = \
	\frac{\left(e^t-y(e^t-1)\right)^2}{e^t}.
	\end{equation*}
	For any two coordinates we obtain 
	\begin{equation*}
	\frac{\left|\left(D_{x}\f_t\right)_{i,i}\right|}
	{\left|\left(D_{x}\f_t\right)_{j,j}\right|} = \left(\frac{e^t-y_i(e^t-1)}{e^t-y_j(e^t-1)}\right)^2,
	\end{equation*}
	whose limit as $t$ goes to infinity is equal to $\frac{1-y_i}{1-y_j}$. 
	Thus, a given bound $y_j \leq u<1$ for $j \in F^+_{01}$ allows us to find a bound $y_i > 1- \delta$ for $i \not\in F^+_{01}$ making this ratio as small as desired for sufficiently large $t$.
\end{proof}

We will also need the following simple lemma.

\begin{lemma}\label{L: domain flow}
	Let $F$ be a terminal face of $\interval^n$, and let $v$ be its initial vertex so that $F=v^+$. Let $0<u<1$, and let $D$ be any neighborhood of $v$ in $F$ containing $L_u(F)$. Then for all $t\geq 0$ we have $L_u(F)\subset \f_t(D)$. Similarly, if $F$ is an initial face of $\interval^n$ and $v$ its initial vertex so that $F=v^-$ and if $D$ is any neighborhood of $v$ in $F$ containing $U_u(F)$ then for all $t\geq 0$ we have $U_u(F)\subset \f_{-t}(D)$.
\end{lemma}

\begin{proof}
	Suppose the first set of hypotheses, and let $y\in L_u(F)$. In particular, this means that if the $i\th$ coordinate of $y$ is a free variables of $F$ we have $y_i<\epsilon$ and otherwise $y_i=1$. As the flow is non-decreasing in each free coordinate and does not move the bound coordinates, we have $\f_{-t}(y)\in L_u(F)$ for $t\geq 0$, and hence $y\in \f_t(L_u(F))\subset \f_t(D)$. The argument for the second set of hypotheses is analogous.
\end{proof}

\subsection{Geometric cochains and intersections}

In this subsection we show that the logistic flow on a manifold cubulated by $|X| \to M$ is compatible with key geometric cochain structures.

\begin{lemma} \label{l:flow preserves transversality}
	If $W \in \cman_{\pf}(M)$, then $\f_t(W)\in \cman_{\pf}(M)$ for all $t \in \R$.
\end{lemma}

\begin{proof}
	A point of some cube $E$ that is in the image of $\f_t(W)$ is the image under the flow diffeomorphism of a point of $E$ in the image of $W$, and the transversality condition is preserved by diffeomorphism.
\end{proof}

\begin{lemma}
	For any $W \in \cman_{\pf}(M)$ and complementary cube $E$,
	\begin{equation*}
	I_M(W, E) = I_M(\f_t(W), E)
	\end{equation*}
	for all $t \in \R$.
\end{lemma}

\begin{proof}
	As $\f_t$ is an orientation-preserving diffeomorphism of $M$, we have $I_M(W, E) = I_M(\f_t(W), \f_t(E))$,
	and, since $\f_t$ restricts to an orientation preserving diffeomorphism of $E$, we also have $\f_t(E) = E$.
\end{proof}

\begin{lemma}
	The logistic flow is well defined on geometric cochains in $C^*_{\Gamma \pf}(M)$.
\end{lemma}

\begin{proof}	
	If ${W}$ and ${W}'$ represent the same element of $C^*_{\Gamma \pf}(M)$, then $ W- W' = A \in Q^*(M)$. By applying $\f_t$ we see that $\f_t( W)-\f_t( W') = \f_t( W - W') = \f_t(A)$, but the flow $\f_t$ is a diffeomorphism and so preserves membership in $Q^*(M)$.	
\end{proof} 


\section{Flow comparison theorem} \label{S: flow comparison theorem}

\subsection{Statement}

We now come to the central result of this paper, which we restate here using the notation developed above.

\begin{theorem} \label{T: main theorem}
	Let $M$ be a cubulated closed manifold. For $W, V \in \cman_\pf(M)$ and $t$ sufficiently large:
	\begin{enumerate}
		\item $\f_t(W)$ and $\f_{-t}(V)$ are transverse and
		\begin{equation*}
		\cI \big( \f_t(W) \times_M \f_{-t}(V) \big) =
		\cI \big( \f_t(W) \big) \sms \cI \big( \f_{-t}(V) \big).
		\end{equation*}
		\item $\f_{-t}(W)$ and $\f_t(V)$ are transverse and
		\begin{equation*}
		\cI \big( \f_{-t}(W) \times_M \f_t(V) \big) =
		(-1)^{|V||W|} \, \cI \big( \f_t(V) \big) \sms \cI \big( \f_{-t}(W) \big),
		\end{equation*}
		where we recall that $|W|$ and $|V|$ are the codimensions of $W$ and $V$ over $M$.
	\end{enumerate}
\end{theorem}

The time needed to flow to obtain the equality between intersection and cup product will vary depending on $W$ and $V$, but on any finite subcomplex we have the following uniformity.

\begin{corollary} \label{C: finite diagram}
	Let $M$ be a closed manifold with smooth cubulation $|X|\to M$, and let $F^*$ be a finitely-generated chain complex with chain map $g:F^*\to C^*_{\Gamma \pf}(M)$.
	Then, there is a $T \in \R$ such that for all $t > T$ the following diagram commutes:
	\begin{equation*}
	\begin{tikzcd} [row sep=tiny]
	& C^*_{\Gamma \pf}(M)^{\otimes 2} \arrow[r, "\cI \otimes \cI"] & \cochains(X)^{\otimes 2} \arrow[dd, "\sms"] \\
	F^*\otimes F^* \arrow[ur, in=180, out=45,"g\otimes g"] \arrow[dr, in=180, out=-45, "\f_t \circ g \; \times_M \; \f_{-t} \circ g"']& & \\
	& C^*_{\Gamma \pf}(M) \arrow[r, "\cI"] & \cochains(X).
	\end{tikzcd}
	\end{equation*}
\end{corollary}

In particular, if $F^*$ is the subcomplex of $C^*_\Gamma(M)$ generated by two cochains $\uW$ and $\uV$, this says that the diagram 
commutes for large enough $t$ starting with the chain $\uW\otimes\uV$ on the left, recapitulating Theorem~\ref{T: main theorem} as a statement about cochains, not just elements of $\cman_\pf(M)$.
There are several other candidates for a useful finitely-generated complex $F^*$, 
with the most desirable being those whose maps to $C^*_{\Gamma \pf}(M)$ induce quasi-isomorphisms.  In that case, Corollary ~\ref{C: finite diagram}
shows that the finitely-generated complex provides a geometric model for cubical cochains, though only as a differential graded associative algebra.
As discussed in the Introduction, we plan to strengthen this connection beyond the associative setting in future work.

The ideal  example of a finitely-generated subcomplex of geometric cochains that is quasi-isomorphic to cubical cochains through intersection
would be generated by a dual complex to the cubulation whose cells are smooth manifolds with corners.
Unfortunately, it is not clear that such dual complexes always exist except in special cases, for example on two-dimensional manifolds.

Another example we can give for a useful $F^*$ comes from considering an oriented manifold $M$ with a smooth triangulation $|K|\to M$ by a finite simplicial complex $K$ with a vertex ordering. 
Using the orientation $\beta_M$ of $M$, every simplicial face inclusion $\sigma\to M$ with $\sigma$ a simplex of $K$ is co-oriented by $(\beta_\sigma,\beta_M)$, where $\beta_\sigma$ is the standard orientation that arises from the ordering of the vertices of $\sigma$. 
Thus we have an inclusion homomorphism $C_*(K)\into C^{m-*}_\Gamma(M)$, where $m = \dim(M)$.
This inclusion is a quasi-isomorphism, as the inclusion of $C_*(K)$ into the smooth singular chain complex $C_*^{ssing}(M)$ of $M$ is a known quasi-isomorphism, the map $C_*^{ssing}(M) \to C_*^\Gamma(M)$ is observed to be a quasi-isomorphism in \cite[Section 10]{Lipy14}, and $C_*^\Gamma(M)=C^{m-*}_\Gamma(M)$ for $M$ closed and oriented (see \cite[Section 12]{Lipy14}). 
We may then modify the triangulation $|K|\to M$ by postcomposing with a map $f \colon M\to M$ that is homotopic to the identity but which shifts each simplex of $K$ into general position with respect to every face inclusion of the cubulation $|X|\to M$.
This can be done as in the proof of Theorem \ref{T: transverse complex}, as given in \cite{FMS-foundations}, using the techniques of \cite[Section 2.3]{GuPo74}.
As $f$ is homotopic to the identity, we obtain the composite quasi-isomorphism $C_*(K)\into C^{m-*}_\Gamma(M) \xr{f} C^{m-*}_\Gamma(M)$, but now with image of the composite in $C^{m-*}_{\Gamma \pf X}(M)$. So by the following diagram,
\begin{equation*}
\begin{tikzcd}[column sep=tiny, row sep=tiny]
C_*(K) \arrow[rr, "f"] \arrow[dr, out=-90, in=180]& & C^{m-*}_{\Gamma}(M) \\
& C^{m-*}_{\Gamma \pf X}(M) \arrow[ur, out=0, in =-90] &
\end{tikzcd}
\end{equation*}
we obtain a quasi-isomorphism from a finite chain complex $C_*(K)\to C^{m-*}_{\Gamma \pf X}(M)$ to which our results can be applied with a single bound.

\bigskip

The proof of Theorem~\ref{T: main theorem} is contained in the following sections. 
As the intersection homomorphism is defined through evaluation on cubes in $X$, the proof 
is ultimately ``local," proceeding over isolated cubes.
We thus make the following definition that will be useful throughout, with $E$ being a cube our main case of interest.

\begin{definition}
	Let $M$ be a manifold without boundary and $E$ and $W$ transverse elements of $\cman(M)$ with $E$ embedded.
	We denote the pull-back $W \times_N E$ by $W_E$. By Theorem~\ref{pullback}, $W_E$ is a c-manifold over $E$, and by \cite{FMS-foundations} the map $W_E\to E$ can be endowed with a pull-back co-orientation induced by the co-orientation of $W\to M$ that does not depend on the orientation of $E$.
\end{definition}

\subsection{Transversality}

We begin with the first claim of Theorem~\ref{T: main theorem}, of transversality. The version of transversality needed for all further claims 
is not just over the entire manifold but also for restrictions over any cube.

\begin{theorem} \label{T: transversality}
	Let $M$ be a cubulated manifold.
	Suppose $W$ and $V$ are in $\cman_\pf(M)$.
	Then, for any cube $E$ in $X$, both $\f_t(W)_E$ and $\f_{-t}(W)_E$ are transverse to $V_E$ for $t$ sufficiently large.
\end{theorem}

\begin{proof}
	Let $\dim(E) = n$, $\dim(W) = w$, and $\dim(V) = v$.
	Let us identify $E$ with $\interval^n$ and omit it from notation, so $\f_t(W)_E$ and $V_E$ become $\f_t(W)$ and $V$.
	It is sufficient to prove the statement for the top stratum of $W$ and $V$, as the argument will apply to any deeper strata.
	We will establish the following statement inducting over $i = 0, \dots, n$.
	\begin{itemize}	
		\item[($\ast$)] There exists a neighborhood $\mathcal N^i$ of $\term_{i}(\interval^n)$ and a $T^i \in \R$ such that $\f_t(W)$ is transverse to $V$ within $\mathcal N^i$ for all $t > T^i$.	
	\end{itemize}
	Since $\interval^n = \term_{n}(\interval^n)$, this will suffice.
	
	For the base case of the induction we consider $\term_0(\interval^n) = \{\underline{1}\}$.
	If $\dim(V) < n$ then, by the assumption that $V$ is transverse to $\interval^n$, there is a neighborhood of $\underline{1}$ that does not intersect $V$.	
	In this case ($\ast$) is fulfilled vacuously.	
	If $\dim(V) = n$, then again the condition is fulfilled vacuously if $V$ does not contain $\underline{1}$, so we assume $\underline{1} \in V$.	
	Because of transversality, the Inverse Function Theorem implies that $V \to \interval^n$ is a local diffeomorphism onto a neighborhood $\mathcal N^0$ of $\underline{1}$ and, therefore, it is transverse to any map therein, so we can take $T^0 = 0$.
	
	Let us establish now the induction step.
	Consider $F \in \term_{i}(\interval^n)$ and notice that the union of $\beta_F = \{\e_i\ |\ i \in F_{01}\}$ and $\beta_{F^-} = \{\e_i\ |\ i \in F^-_{01}\}$ trivialize the tangent bundle of $\interval^n$.
	
	Choose $\delta > 0$ small enough so that the closed $L^\infty$ neighborhood $\overline N_\delta(F)$ of $F$ is such that for any $y \in V \cap \overline N_\delta(F)$ the affine space $T_yV$ is transverse to the span of $\beta_F = \{\e_i\ |\ i \in F_{01}\}$ at $y$.
	
	Similarly, consider a $\zeta > 0$ small enough so that the closed $L^\infty$ neighborhood $\overline N_\zeta(F^-)$ of $F^-$ is such that either $W \cap \overline N_\zeta(F^-) = \emptyset$ or for any $x \in W \cap \overline N_\zeta(F^-)$ the affine space $T_xW$ is transverse to the span of $\beta_{F^-} = \{\e_i\ |\ i \in F^-_{01}\}$ at $x$.
	In the latter case, for every such $x$, the transversality implies that $T_xW$ projects surjectively onto $F^-$; in other words, the projection of $W \cap \overline N_\zeta(F^-)$ to $F^-$ is a submersion onto its image. Therefore, we may choose continuously with $x$ a subset of $T_xW$ of the form $\beta_x = \{\e_j + v_j\ |\ j \in F_{01}\}$ such that $v_j$ is in the span of $\beta_{F^-}$.
	As $W \cap \overline N_\zeta(F^-)$ is compact, the $v_j$ have bounded norm.
	
	By possibly making $\delta$ smaller we can choose $u \in (0, 1)$ such that $\overline N_\delta(F) \setminus \overline N_\delta L_u(F)$ is contained in $\mathcal N^{i-1}$.
	
	By Lemma \ref{L: flow to initial and terminal faces} we may choose $t > T^{i-1}$ sufficiently large so that all points in $\f_t(W) \cap V \cap \overline N_\delta L_u(F)$ are of the form $y = \f_t(x)$ for some $x \in W \cap \overline N_\zeta(F^-)$.
	
	We use Lemma~\ref{L: jacobian ratios} to deduce that the push forward of the span of $\beta_x$ along $D\,\f_t$ is as close as desired to the span of $\beta_F$ at $y$ and is therefore transverse to $T_y V$.
	The induction step is completed by taking $\mathcal N^i$ to be the union over all terminal faces of $N_\delta(F)$ and $T^i$ the maximum value of their associated $t$.
\end{proof}

If $M$ is closed, then any cubulation of $M$ is by a finite number of cubes.
Applying this theorem for each top-dimensional cube yields that if $W$ and $V$ are in $\cman_\pf(M)$ then both $\f_t(W)$ and $\f_{-t}(W)$ are transverse to $V$ for $t$ sufficiently large.
We have the following consequence of this theorem related to Theorem~\ref{T: main theorem}.

\begin{corollary} \label{C: transversality}
	Let $M$ be a cubulated closed manifold.
	If $W$ and $V$ are in $\cman_\pf(M)$ then for $t$ sufficiently large  $\f_t(W)$ and $\f_{-t}(V)$, as well $\f_{-t}(W)$ and $\f_t(V)$, are transverse over $M$.
\end{corollary}

\begin{proof}
By Theorem \ref{T: transversality} there is a $t$ large enough so that $\f_{2t}(W)$ and $V$ are transverse for all larger $t$. Now apply $\f_{-t}$ to both terms and that transversality is preserved by composition with a diffeomorphism. Similarly, we can apply $\f_t$ to $\f_{-2t}(W)$ and $V$ for large enough $t$. Now choose $t$ large enough to do both. 
\end{proof}

\subsection{Locality}

On a cubulated manifold the logistic vector field is compatibly defined across cubes by Lemma~\ref{L: f is natural}.
In this subsection we show that we can analyze the pull-back product over $M$ of cochains $\f_t(W)$ and $V$ locally -- that is, cube-by-cube.

\begin{lemma} \label{L: little trans lemma}
	Let $V$ and $W$ be c-manifolds over $M$ and $S$ a submanifold without boundary of $M$.
	Suppose $V$, $W$, and $S$ are pairwise transverse. 
	If $V_S$ and $W_S$ are transverse over $S$ then $W \times_M V$ is transverse to $S$ over $M$. Moreover,	$W_S \times_S V_S \cong (W \times_M V) \times_M S$.
\end{lemma}

We will apply this when $S$ is the interior of a cube in a cubical structure.

\begin{proof}		
	The last statement of the lemma follows from identifying both $W_S \times_S V_S$ and $ (W \times_M V) \times_M S$ with
	the subspace of triples of $(x,y,z) \in W \times V \times S$ such that $r_W(x) = r_V(y) = r_S(z)$. 
	
	This observation at the level of tangent spaces	also gives rise to the first statement, recalling that the tangent bundle of the fiber product is the fiber product of the tangent bundles.
	Indeed, transversality of two maps at a point where they coincide is defined locally by surjectivity of the map $(\vec a,\vec b)\to D_xr_W(\vec a)-D_yr_V(\vec b)$ from the direct sum of tangent spaces of the domain points to that of the ambient manifold. By exactness, this surjectivity is equivalent to the kernel having the appropriate dimension, but the kernel is precisely the tangent space of the	fiber product.
	Thus, it follows from a short computation that two maps are transverse if and only if the the fiber-product of tangent spaces over any point has codimension equal to the sum of the codimensions of the tangent spaces.
	As in the ``global" case, the fiber product of $T_{(x,z)} V_S$ and $T_{(y,z)} W_S$ over $T_zS$ coincides with that of $T_{(x,y)} (W \times_M V)$ and $T_z S$ over $T_z M$, as both are identified with the same subspace of $T_x W \times T_y V \times T_z S$. But the first pull-back is transverse by assumption, so this subspace has codimension equal to the sum of the codimensions of $S$, $V$, and $W$ in $M$.
	This codimension is also what is required for the second pull-back to be transverse, yielding the result.
\end{proof}

\begin{proposition} \label{P: locality}
	Let $M$ be a cubulated closed manifold.
	Suppose $W$ and $V$ are in $\cman_\pf(M)$.
	For $t$ sufficiently large, the pull-back $\f_t(W) \times_M V$ is transverse to the cubulation, and 
	\begin{equation*}
	\f_t(W) \times_M V = \bigcup_{E \in X} {\f_t(W_E)} \times_E V_E
	\end{equation*}
	as spaces, with each ${\f_t(W_E)} \times_E V_E$ a manifold with corners.
\end{proposition}

\begin{proof}
	By Lemma~\ref{l:flow preserves transversality}, the c-manifold $\f_t(W)$ remains transverse to the cubulation for every $t$.
	By Theorem~\ref{T: transversality}, over each $E$ in the cubulation, $\f_t(W)_E$ is transverse to $V_E$ for $t$ sufficiently large. Taking the largest
	such $t$, we can then apply Lemma~\ref{L: little trans lemma} to obtain transversality of the pull-back of $\f_t(W)$ and $V$ to the entire cubulation. 
	
	Tautologically, $ \f_t(W) \times_M V = \bigcup_{E \in X} {\f_t(W)}_E \times_E V_E$.
	But again using that the logistic flow respects the cubulation, ${\f_t(W)}_E = {\f_t(W_E)}$, which gives the decomposition.
	That this is itself a manifold with corners follows from the identification by Lemma \ref{L: little trans lemma} of $ {\f_t(W_E)} \times_E V_E$ with $(\f_t(W) \times_M V)\times_M E$, the latter being a manifold with corners by our transversality result.
\end{proof}

For the following lemma, recall Definition \ref{D: intersection number}.

\begin{lemma} \label{L: intersection signs}
	Let $W$ and $V$ be co-oriented c-manifolds over a closed manifold $M$ that are transverse to a cubulation $|X| \to M$, and let $E$ be a cube whose dimension
	is the sum of the codimensions of $W$ and $V$ in $M$. Then, for sufficiently large $t$, $\f_t(W_E)$ and $V_E$ are complementary in $E$ and the value of $\cI(\f_t(W) \times_M V)$ on $E$
	is equal to $I_E(\f_t(W_E), V_E)$.
\end{lemma}

\begin{proof}
	We assume that $t$ is as large as needed for Proposition~\ref{P: locality}.
	In this case, both $(r_{\f_t(W) \times_M V})^{-1}(E)$ and $\rm{Int}_E(\f_t(W_E), V_E)$ consist of points $(x,y,z)\in W \times V \times E$ with $\f_t \circ r_W (x)= r_V (y)= \iota_E(z)$. 
	We check that this is an isomorphism of signed sets.
		
	Consider one such intersection point. We first compute its contribution to $I_E(\f_t(W_E), V_E)$. As $\f_t(W)_E$ and $V_E$ are transverse and of complementary dimension in $E$, these spaces are embedded in $E$ near the intersection point.
	Moreover, $\f_t W$ and $V$ are embedded in $M$ near such a point, since a nontrivial $\vec v$ in the kernel of the derivative of $r_{\f_t (W)}$ would imply $(\vec v,0)\in TW\times_{TM} TE$, and it maps to $0$ in $T_M$, producing a nontrivial kernel for the derivative of $r_{\f_t(W_E)}=r_{\f_t(W)_E}$. Similarly for $V$ and $V_E$.
	The codimensions of $\f_t (W_E)$ and $V_E$ in $E$ agree with the codimensions of $W$ and $V$ in $M$, and their normal bundles in $E$ are the restriction of the normal bundles of $\f_t(W)$ and $V$ in $M$ where they are embedded.
	Thus we can use normal co-orientations and identify the normal co-orientation $\beta_{\f_t(W)}$ of $\f_t(W)_E$ in $E$ with the normal co-orientation of $\f_t(W)$ in $M$ and similarly for $V_E$ and $V$.
	Then the sign of the intersection is $1$ if $\beta_W \wedge \beta_V$ agrees with the orientation of $E$ and is $-1$ otherwise.
	 
	Now consider the pullback $(r_{\f_t(W) \times_M V})^{-1}(E)$. 
	As $\f_t( W)$ and $V$ are embedded near $x$ and $y$, we co-orient $\f_t(W) \times_M V = \f_t(W)\cap V$ at $z$ again by $\beta_{\f_t(W)} \wedge \beta_V$ according with the properties of co-orientations of fiber products of embeddings; see Theorem \ref{T: pull-back co-or}.
	But now again the contribution to $\cI(\f_t(W) \times_M V)$ is $+1$ or $-1$ as this pair agrees or not with the orientation of $E$.
\end{proof}

\begin{corollary} \label{C: intersection signs}
	With the assumptions of Lemma \ref{L: intersection signs}, for sufficiently large $t$, $\f_t(W_E)$ and $\f_{-t}(V_E)$ are complementary in $E$ and the value of $\cI(\f_t(W) \times_M \f_{-t}(V))$ on $E$ is equal to $I_E(\f_t(W_E), \f_{-t}(V_E))$.
\end{corollary}

\begin{proof}
	By the preceding lemma, there is a $t$ large enough that $\f_{2t}(W_E)$ and $V_E$ are complementary in $E$ and the value of $\cI(\f_{2t}(W) \times_M V)$ on $E$ is equal to $I_E(\f_{2t}(W_E), V_E)$.
	Now apply the diffeomorphism $\f_{-t}$ to $\f_{2t}(W_E)$ and $V_E$.
	As $\f_{-t}$ preserves orientations and co-orientations, the corollary follows.
\end{proof}

\subsection{Graph-like neighborhoods}

We now focus on the local structure of a transverse c-manifold over $\interval^n$ around intersection points with faces of complementary dimension.
The Implicit Function Theorem guarantees that the following subspaces occur naturally in this setting.

\begin{figure}[!h]
	\centering
	\begin{subfigure}{.32\textwidth}
		\includegraphics[scale=.7]{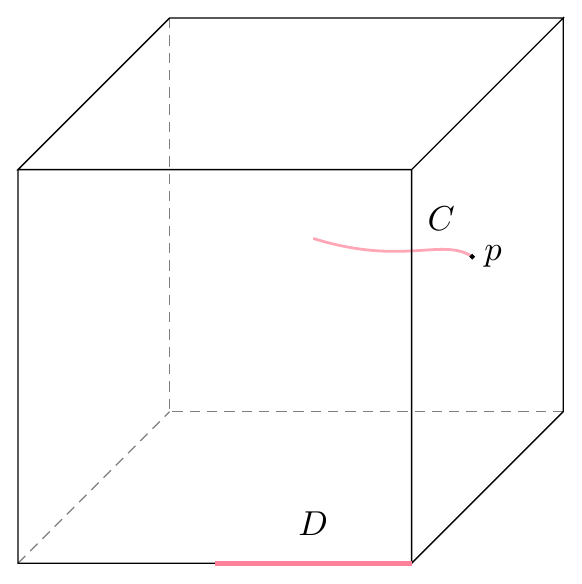}
		\hfill
	\end{subfigure}
	\begin{subfigure}{.32\textwidth}
		\hfill
		\includegraphics[scale=.7]{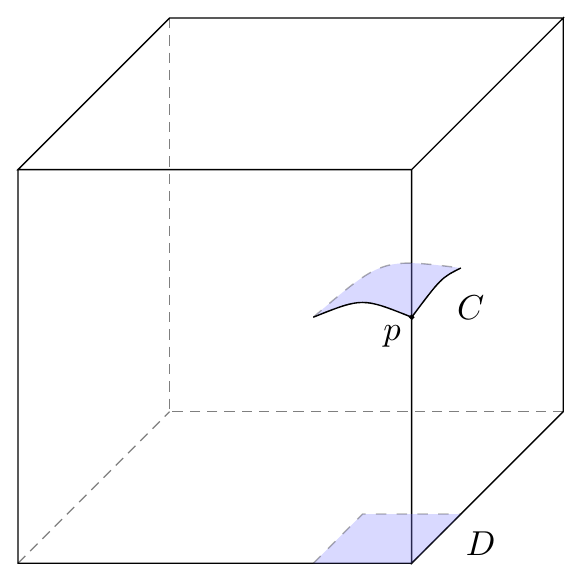}
	\end{subfigure}
	\caption{Examples of graph-like neighborhoods with $p$ in the interior of faces $F=(\emptyset, \{2,3\}, \{1\})$ and $F=(\{2\}, \{3\}, \{1\})$, respectively. In the diagram on the right, $F^-\times F^+$ is difficult to depict so we draw the domain $D$ in the bottom face of the cube, which is also complementary to $F$. }
	\label{F: graph like neighborhood}
\end{figure}

\begin{definition} \label{D: graph-like}
	Let $F$ be a face of $\interval^n$ and $p$ a point in its interior.
	Consider the $F$-decomposition $\interval^n \cong F^- \times F \times F^+$ and its canonical projections $\pi_F^{\perp}$ to $F^- \times F^+$ and $\pi_F$ to $F$.
	A set $C$ is said to be a \textbf{graph-like neighborhood} of $p$ if 
	\begin{enumerate}
		\item $C$ is the graph of a smooth map $s: D \to F$, where $D$ is an open neighborhood of $\pi_F^\perp(p)$ in $F^- \times F^+$ such that $D$ intersects only the faces of $F^- \times F^+$ containing $\pi_F^\perp(p)$ and is bounded away from the other faces of $F^- \times F^+$,
		
		\item $s(\pi_F^\perp(p)) = p$,
		
		\item $C$ is bounded away from the faces of $F$ that do not contain $p$, and
		
		\item $C$ is transverse to all faces of $\interval^n$.
	\end{enumerate}
\end{definition}

In general $C$ is not actually a neighborhood of $p$ in $\interval^n$ as $\dim(C)=n-\dim F$, with an extreme case being $F =\interval^n$ in which case $C$ is just a point in the interior of $F$.
For any graph-like neighborhood $C$ of a point in the interior of a face $F$, there exists $r, u \in (0,1)$ such that $C \subseteq N_rL_u(F)$ and $C \subseteq N_rU_u(F)$.

Since the logistic flow can be expressed independently in each coordinate, it preserves the property of being a graph-like neighborhood.

Let $W$ be a c-manifold over $\interval^n$.
Assume that for any $(n-\dim W)$-face $F$ of $\interval^n$, the set $W \cap F$ consists of a finite number of points near which $W$ is locally embedded.
By the Implicit Function Theorem, the discrete set $W \times_{\interval^n} F$ can be used to parameterize a collection of graph-like neighborhoods of the points in $W \cap F$.
Recall the definition of reciprocal faces of $\interval^n$, Definition~\ref{D: reciprocal}.

\begin{lemma} \label{L: flow intersection for graph-like nbhds}
	Let $F$ and $F^\prime$ be faces of $\interval^n$ of complementary dimension and let $C$ and $C^\prime$ be graph-like neighborhoods of points in their respective interiors.
	If the pair $F$ and $F^\prime$ is reciprocal, then for $t$ sufficiently large $\f_t(C) \cap \f_{-t}(C^\prime)$ is a single point.
	If the pair is not reciprocal then for $t$ sufficiently large this intersection is empty.
\end{lemma}

\begin{figure}[!h]
	\includegraphics[scale=.65]{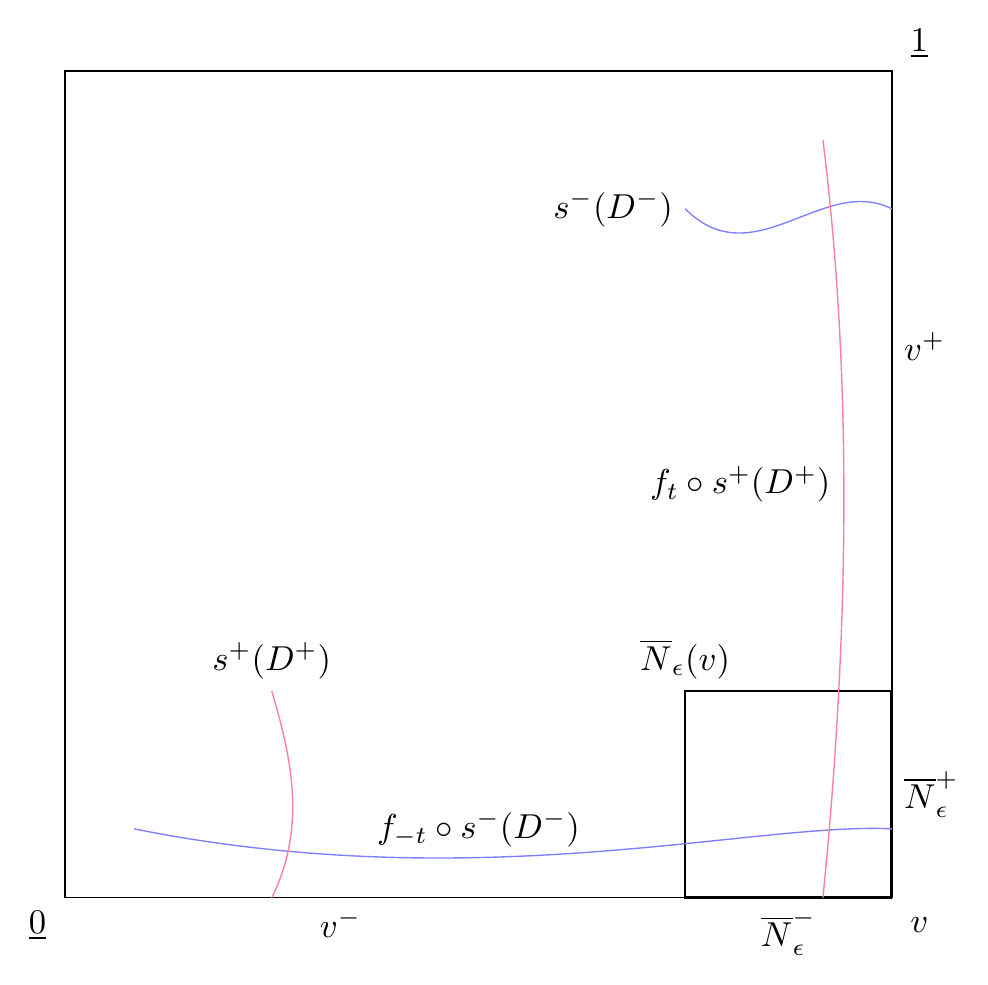}
	\caption{Proof of Lemma~\ref{L: flow intersection for graph-like nbhds}.}
	\label{F: intersection}
\end{figure}

\begin{proof}
	We refer the reader to Figure \ref{F: intersection} to accompany the proof.

	First suppose the pair $(F,F^\prime)$ is reciprocal.
	This is the case if and only if there is a vertex $v$ with $v^- = F$ and $v^+ = F^\prime$.
	Let $s^\pm \colon D^\pm \subseteq v^\pm \to \interval^n = v^- \times v^+$ be sections of $\pi_{v^\pm}^\perp$ such that $C = s^+(D^+)$ and $C^\prime = s^-(D^-)$.	
	Let us consider a closed $L^\infty$ $\epsilon$-neighborhood $\overline N_\epsilon(v)$ of $v$ in $\interval^n$ with $\epsilon$ small enough so that $L_\epsilon(v^\pm) \subset D^\pm$.
 	Write $\overline N_\epsilon^\pm = \overline N_\epsilon(v) \cap v^\pm$.
	Let us assume $t$ large enough so that $\norm{\f_{\pm t}(s^\pm(v)) - v} < \epsilon$, which is possible because $s^\pm(v)$ are in the interiors of $v^\pm$ by assumption and $v$ is the terminal vertex of $v^-$ and the initial vertex of $v^+$.
	Let $\overline N_t^\pm = \overline N_\epsilon(v) \cap \big(\f_{\pm t}\circ s^\pm (D^\pm)\big)$. As $\f_t$ flows each coordinate independently, $\f_t(C)$ remains a graph of $\f_t(D)$ for all $t$. Therefore, by Lemmas \ref{L: flow to initial and terminal faces}, \ref{L: jacobian ratios}, and \ref{L: domain flow}, for large enough $t$ the set $\overline N_t^\pm$ will be as small a perturbation as desired of $\overline N_\epsilon^\pm$ and additionally the tangent space at each point of $\overline N_t^\pm$ will be as small a perturbation as desired of the plane containing $v^\pm$. 
	Therefore, since $\overline N_\epsilon^+ \cap \overline N_\epsilon^- = \{v\}$, the stability of transverse intersections implies that there exists a unique $y_t \in \overline N_t^+ \cap \overline N_t^-$. The lemma follows in this case since there cannot be intersections outside $\overline N_\epsilon(v) = \overline N_\epsilon(v^+) \cap \overline N_\epsilon(v^-)$ since, by Lemma~\ref{L: flow to initial and terminal faces}, we can assume $t$ large enough so that $\f_t(C) \subseteq \overline N_\epsilon(v^+)$ and $\f_{-t}(C^\prime) \subseteq \overline N_\epsilon(v^-)$.

	Next we assume that the pair $(F,F^\prime)$ is not reciprocal (but still of complementary dimension). In particular, $F$ is not initial or $F^\prime$ is not terminal. 
	Let $r,u,r^\prime,u^\prime \in (0,1)$ such that $C \subseteq N_rU_u(F)$ and $C^\prime \subseteq N_{r^\prime}L_{u^\prime}(F^\prime)$, which is possible by the definition of graph-like neighborhoods.
	First suppose $F$ is not initial. 
	By Lemma~\ref{L: flow to initial and terminal faces} we have that for any $\epsilon > 0$ and all $t$ sufficiently large $\f_{2t}(C) \subseteq N_\epsilon(F^+)$.
	As $F$ is not initial, $\dim F^+ < n-\dim F$, so for large enough $t$ we have $\f_{2t}(C)$ contained in a neighborhood of the $n-\dim F-1$ skeleton of $\interval^n$. But now $\dim(C^\prime) = n-\dim F^\prime = \dim F$, and $C^\prime$ is supposed by definition to be transverse to $\interval^n$ and bounded away from the faces of $\interval^n$ it does not intersect.
	Thus for sufficiently large $t$ we have $\f_{2t}(C)\cap C^\prime = \emptyset$, and as $\f_{-t}$ is a diffeomorphism, applying it to both terms we obtain $\f_{t}(C) \cap \f_{-t} (C^\prime) = \emptyset$.
	The argument if $F^\prime$ is not terminal is analogous.
\end{proof}

The preceding lemma showed that for large enough $t$ the intersection $\f_t(C) \cap \f_{-t}(C^\prime)$ is empty unless the pair $(F,F^\prime)$ is reciprocal, in which case the intersection is a single point.
The next lemma determines how the sign of that intersection point, when it exists, depends on the co-orientations of $C$ and $C^\prime$. 

Recall that the faces of $\interval^n$ have a canonical orientation induced from the order of the basis $\{\e_1, \dots \e_n\}$.
Therefore, a \textbf{compatible co-orientation} can be defined for any graph-like neighborhood $C$ to be the normal co-orientation induced from the orientation of the face $F$ in complementary dimensions that $C$ intersects.
In other words, the compatible co-orientation of $C$ is determined by $\cI(C)(F) = +1$. Note that $C$ is diffeomorphic to its domain $D$ and so it is a manifold with corners. In general it will not be properly embedded so we abuse notation in writing $\cI(C)$, but the meaning should remain clear. 

Recall that for $v \in \vertices(\interval^n)$ the shuffle sign $\sh(v)$ is the sign of the shuffle permutation of $\{1, \dots, n\}$ placing the ordered free variables in $v^-$ before those in $v^+$. 

\begin{lemma} \label{L: sign for graph-like nbhds intersection}
	Let $v \in \vertices(\interval^n)$.
	If $C$ and $C^\prime$ are compatibly co-oriented graph-like neighborhoods of points respectively in the interiors of $v^-$ and $v^+$, then, for $t$ sufficiently large,
	\begin{equation}
	\label{E: sign flow intersection graph-like 1}
	\cI \big( \f_t(C) \times_{\interval^n} \f_{-t}(C^\prime) \big)
	\big( [\underline{0}, \underline{1}]^{\otimes n} \big) = \sh(v). \\
	\end{equation}
\end{lemma}

\begin{proof}
	We start by noticing that the hypotheses imply that $C$ and $C'$ have complementary dimensions in $\interval^n$ since $v^-$ and $v^+$ do.

	By Lemma \ref{L: flow intersection for graph-like nbhds}, the fiber product $\f_t(C) \times_{\interval^n} \f_{-t}(C^\prime)$ is a single point.
	Since the logistic flow is an orientation-preserving diffeomorphism, it preserves co-orientation as well.
	In terms of ordered bases, by Theorem \ref{T: pull-back co-or} the co-orientation of the fiber product is determined by the orientation of the normal bundle of the intersection point determined by the ordered concatenation $\beta_- \cup \beta_+$, where $\beta_\pm$ is an ordered basis representing the orientation of $v^\pm$. From this the claim follows.
\end{proof}

\subsection{The proof of Theorem \ref{T: main theorem}}

In this section we prove our main theorem; let us first recall its statement.
Let $M$ be a cubulated closed manifold.
For $W, V \in \cman_\pf(M)$ and $t$ sufficiently large:
\begin{enumerate}
\item $\f_t(W)$ and $\f_{-t}(V)$ are transverse and
\begin{equation*}
\cI \big( \f_t(W) \times_M \f_{-t}(V) \big) =
\cI \big( \f_t(W) \big) \sms \cI \big( \f_{-t}(V) \big).
\end{equation*}
\item $\f_{-t}(W)$ and $\f_t(V)$ are transverse and
\begin{equation*}
\cI \big( \f_{-t}(W) \times_M \f_t(V) \big) =
(-1)^{|W||V|} \, \cI \big( \f_t(V) \big) \sms \cI \big( \f_{-t}(W) \big),
\end{equation*}
where $|W||V|$ is the product of the codimensions of $W$ and $V$ over $M$.
\end{enumerate}

The transversality statement was proven as Theorem~\ref{T: transversality}.
By Proposition~\ref{P: locality} and Corollary~\ref{C: intersection signs}, it suffices to consider pull-backs over an arbitrary $n$-face, where $n$ is the sum of the codimensions of $V$ and $W$ over $M$.
We identify this face with $\interval^n$ and, by abuse, denote the pull-backs of $V$ and $W$ over this face also by $V$ and $W$. 

Let us consider (1) first.
Since $W$ is transverse to all faces of $\interval^n$ in particular we have that its intersection with faces of complementary dimension are discrete.
By the Implicit Function Theorem, for any such $F$ there are local neighborhoods in $W$ of the points in $W \cap F$ which are graph-like.
Furthermore, the pull-back $W \times_{\interval^n} F$ can be used to parameterize these graph-like neighborhoods, which we assume equipped with the co-orientation induced from $W$.
We use this co-orientation to endow $W \times_{\interval^n} F$ with a sign function,
sending an element to $+1$ if the orientation of the normal bundle to the graph-like neighborhood given by its co-orientation agrees 
with the orientation of $F$ and to $-1$ if not.
A similar situation applies when considering $V$ and faces of dimension $\dim W$.

Let $\Ginit_W$ (resp. $\Gterm_V$) be the union of the parameterizing sets of these neighborhoods over initial (resp. terminal) faces. Explicitly, 
\begin{equation*}
\Ginit_W\ =
\bigcup_{F \in \init_{\dim V}(\interval^n)} W \times_{\interval^n} F
\qquad \text{ and } \qquad
\Gterm_W\ =
\bigcup_{F^\prime \in \term_{\dim W}(\interval^n)} V \times_{\interval^n} F^\prime.
\end{equation*}
Moreover, let
\begin{equation*}
Z_W^\init =\, W\ \setminus \bigcup_{\Ginit_W} C
\qquad \text{ and } \qquad
Z_V^\term =\, V\ \setminus \bigcup_{\Gterm_V} C^\prime.
\end{equation*}

If $\dim V=n$, then $Z_W^\init=\emptyset$, and similarly if $\dim V=0$ then $Z_V^\term=\emptyset$. 
So we focus on $\dim V<n$.
We notice that $Z_W^\init$ is contained in the complement of a neighborhood $\init_{\dim V}(\interval^n)$. 
Therefore, no point in $Z_W^\init$ can have $n-\dim V$ (or more) coordinates equal to $0$, and in fact we can assume there is an $\eta>0$ 
so that every point of $Z_W^\init$ has more than $\dim V$ coordinates that are larger than $\eta$. 
So, for any $\epsilon>0$ we can choose a large enough $t$ to ensure that every point of $\f_t(Z_W^\init)$ has more than $\dim V$ coordinates greater than 
$1-\epsilon$. In other words, $\f_t(Z_W^\init)$ can be contained in any given neighborhood of $\term_{n-\dim V-1}(\interval^n)$. 

As $V$ maps properly and transversely to $\interval^n$, there is a neighborhood of the $n-\dim(V)-1$ skeleton of $\interval^n$ that is disjoint from $V$, and so we have shown that $\f_t(Z_W^\init) \cap V = \emptyset$ for large enough $t$.
Replacing $t$ with $2t$ in the preceding sentence and then applying $\f_{-t}$ to $\f_{2t}(Z_W^\init)$ and $V$ we obtain that $\f_t(Z_W^\init) \cap \f_{-t}(V) = \emptyset$ for large enough $t$.
Similarly, we can find $t$ large enough so that $\f_{t}(W) \cap \f_{-t}(Z_V^\term) = \emptyset$.
This shows that
\begin{equation}\label{E: reduction to C}
\f_t(W) \times_{\interval^n} \f_{-t}(V) = \f_t \left( \bigcup_{\Ginit_W} C \right) \times_{\interval^n} \f_{-t} \left( \bigcup_{\Gterm_V} C^\prime \right).
\end{equation}

Lemma \ref{L: flow intersection for graph-like nbhds} now implies that for $t$ large enough the discrete set \eqref{E: reduction to C} is in bijection with the discrete set
\begin{equation*}
S\ \ =\!\! \bigcup_{v \in \vertices(\interval^n)} \left(W \times_{\interval^n} v^-\right) \times \left(V \times_{\interval^n} v^+\right) \ \subseteq\ \Ginit_W \times \Gterm_V.
\end{equation*}
Considering $\left(W \times_{\interval^n} v^-\right)$ and $\left(V \times_{\interval^n} v^+\right)$ as signed sets as just above, we define a sign function on $S$ by sending a pair $(\xi, \eta) \in S$ to the product of the sign of $\xi \in \left(W \times_{\interval^n} v^-\right)$, the sign of $\eta \in \left(V \times_{\interval^n} v^+\right)$, and the shuffle sign of $v$.
As described in Lemma~\ref{L: sign for graph-like nbhds intersection}, this sign function makes the bijection between $\f_t(W) \times_{\interval^n} \f_{-t}(V)$ and $S$ sign-preserving.

As we now have a signed bijection of $\f_t(W) \times_{\interval^n} \f_{-t}(V)$ and $S$, comparing with the Serre diagonal in Proposition~\ref{P: diagonal in terms of vertices} gives
\begin{align*}
\cI (\f_t(W) \times_{\interval^n} \f_{-t}(V)) \left([\underline{0}, \underline{1}]^{\otimes n}\right) = &
\sum_{v \in \vertices(\interval^n)} \sh(v) \cdot \cI(\f_t(W))(v^-) \cdot \cI(\f_{-t}(V))(v^+) \\ = & \
\cI(\f_t(W)) \sms \cI(\f_{-t}(V)) \left([\underline{0}, \underline{1}]^{\otimes n}\right)
\end{align*}
for $t$ sufficiently large.

To prove (2), we first interchange the roles of $V$ and $W$ in (1) to obtain
\begin{equation*}
\cI \big( \f_t(V) \times_M \f_{-t}(W) \big) =
\cI \big( \f_t(V) \big) \sms \cI \big( \f_{-t}(W) \big).
\end{equation*}
But now $\f_t(V) \times_M \f_{-t}(W)=(-1)^{|V||W|}\f_{-t}(W) \times_M \f_{t}(V)$ by Theorem~\ref{T: pull-back co-or}.
\hfill\qedsymbol

\bibliographystyle{alpha}
\bibliography{flows}

\end{document}